\documentclass{article}
\usepackage{array,amsmath,amsthm}
\numberwithin{equation}{section}
\usepackage{amssymb}
\usepackage{indentfirst}
\usepackage{geometry}
\allowdisplaybreaks
\usepackage{xcolor}
\begin{document}

\newtheorem{thm}{Theorem}[section]
\newtheorem{cor}[thm]{Corollary}
\newtheorem{prop}[thm]{Proposition}
\newtheorem{conj}[thm]{Conjecture}
\newtheorem{lem}[thm]{Lemma}
\newtheorem{Def}[thm]{Definition}
\newtheorem{rem}[thm]{Remark}
\newtheorem{prob}[thm]{Problem}
\newtheorem{ex}{Example}[section]

\newcommand{\be}{\begin{equation}}
\newcommand{\ee}{\end{equation}}
\newcommand{\ben}{\begin{enumerate}}
\newcommand{\een}{\end{enumerate}}
\newcommand{\beq}{\begin{eqnarray}}
\newcommand{\eeq}{\end{eqnarray}}
\newcommand{\beqn}{\begin{eqnarray*}}
\newcommand{\eeqn}{\end{eqnarray*}}
\newcommand{\bei}{\begin{itemize}}
\newcommand{\eei}{\end{itemize}}

\newcommand{\pa}{{\partial}}
\newcommand{\V}{{\rm V}}
\newcommand{\R}{{\bf R}}
\newcommand{\K}{{\rm K}}
\newcommand{\e}{{\epsilon}}
\newcommand{\tomega}{\tilde{\omega}}
\newcommand{\tOmega}{\tilde{Omega}}
\newcommand{\tR}{\tilde{R}}
\newcommand{\tB}{\tilde{B}}
\newcommand{\tGamma}{\tilde{\Gamma}}
\newcommand{\fa}{f_{\alpha}}
\newcommand{\fb}{f_{\beta}}
\newcommand{\faa}{f_{\alpha\alpha}}
\newcommand{\faaa}{f_{\alpha\alpha\alpha}}
\newcommand{\fab}{f_{\alpha\beta}}
\newcommand{\fabb}{f_{\alpha\beta\beta}}
\newcommand{\fbb}{f_{\beta\beta}}
\newcommand{\fbbb}{f_{\beta\beta\beta}}
\newcommand{\faab}{f_{\alpha\alpha\beta}}

\newcommand{\pxi}{ {\pa \over \pa x^i}}
\newcommand{\pxj}{ {\pa \over \pa x^j}}
\newcommand{\pxk}{ {\pa \over \pa x^k}}
\newcommand{\pyi}{ {\pa \over \pa y^i}}
\newcommand{\pyj}{ {\pa \over \pa y^j}}
\newcommand{\pyk}{ {\pa \over \pa y^k}}
\newcommand{\dxi}{{\delta \over \delta x^i}}
\newcommand{\dxj}{{\delta \over \delta x^j}}
\newcommand{\dxk}{{\delta \over \delta x^k}}

\newcommand{\px}{{\pa \over \pa x}}
\newcommand{\py}{{\pa \over \pa y}}
\newcommand{\pt}{{\pa \over \pa t}}
\newcommand{\ps}{{\pa \over \pa s}}
\newcommand{\pvi}{{\pa \over \pa v^i}}
\newcommand{\ty}{\tilde{y}}
\newcommand{\bGamma}{\bar{\Gamma}}

\title {The Bonnet-Myers theorem on Finsler manifolds with integral weighted Ricci curvature bounds
\footnote{The first author is supported by the National Natural Science Foundation of China (12371051, 12141101, 11871126)}}
\author{ Xinyue Cheng, \ Liulin Liu}
\date{}
\maketitle

\begin{abstract}
In this paper, we derive some new relative volume comparison theorems and Bishop-Gromov volume comparisons on Finsler metric measure manifolds, all of which are controlled by the integral weighted Ricci curvature. In particular, we establish a Bishop-Gromov volume comparison theorem for nonconcentric balls. Based on these, we prove a theorem of Bonnet-Myers type on Finsler metric measure manifolds with integral weighted Ricci curvature bounds.
\end{abstract}

\section{Introduction}\label{Introd}

Finding the topological, geometrical or analytical properties induced by curvature bounds is an important topic in Riemannian geometry. The classical Bonnet-Myers theorem in Riemannian geometry says that, if the Ricci curvature of an  $n$-dimensional complete connected Riemannian manifold $(M,g)$ satisfies ${\rm Ric} \geq (n-1)H$ for some constant $H >0$, then manifold $M$ is in fact compact with finite fundamental group and its diameter is at most $\frac{\pi}{\sqrt{H}}$ \cite{Myers}. Since it occurs, there have been many subsequent generalizations of this result (e.g. see \cite{WeiWylie, WuJY2}).

Many geometric problems in Riemannian geometry lead naturally to integral curvatures. For example, the isoperimetric inequalities, the Sobolev constant estimate, the gradient estimate, the heat kernel estimate, geometric variational problems, and so on. Comparing with pointwise curvature lower bound, integral curvature bound is a global condition.  From 1990s, Bonnet-Myers theorem was generalized to the Riemannian metric measure spaces or the smooth metric measure spaces with integral curvature bounds.   In the fundamental work \cite{PeterWei}, the important Laplacian comparison and volume comparison are generalized to Riemannian manifolds with integral Ricci curvature lower bounds. Further,  the Cheeger-Colding-Naber theory was successfully extended to the manifolds with integral Ricci curvature bounds \cite{PeterW2}. Based on these, Petersen-Sprouse \cite{PeterSprouse} investigated how to generalize the classical diameter bound estimate for manifolds with positive Ricci curvature to the situation where the manifolds have integral Ricci curvature bound and got a rough diameter bound. Further, Aubry \cite{Aubry} further refined this diameter bound by applying the comparison results in \cite{PeterWei,PeterW2} to star-shaped domains. Besides, Sprouse \cite{Spro} derived a diameter upper bound and proved the finiteness of fundamental group for complete Riemannian manifolds with integral Ricci curvature bounds and satisfying ${\rm Ric}\geq (n-1)k \ (k\leq 0)$.
Moreover, Wu \cite{WuJY1} got diameter estimates, eigenvalue estimates, and volume growth estimates on smooth metric measure spaces with their normalized integral smallness for Bakry-\'{E}mery Ricci tensor. Further, inspired by Aubry's work, Li-Wu-Zheng \cite{LiWuYu} got a new theorem of Bonnet-Myers type on Riemannian metric measure space $(M, g, e^{-f} {\rm Vol}_{g})$ with the integral Bakry-\'{E}mery Ricci tensor bound which sharpens  Wu's result in \cite{WuJY1}.

Finsler geometry is just Riemannian geometry without the quadratic restriction \cite{Chern}. It is natural to derive the Bonnet-Myers theorem on Finsler metric measure manifolds. Bao-Chern-Shen \cite{BaoChernShen}  proved the Bonnet-Myers theorem on a forward geodesically complete connected Finsler manifold $(M, F)$ under a uniform positive lower bound on the Ricci curvature. Concretely, if ${\rm Ric}\geq (n-1)K >0$, they proved that $M$ is in fact compact and has finite fundamental group, and the diameter of $M$ is at most $\frac{\pi}{\sqrt{K}}$ (\cite{BaoChernShen}, Theorem 7.7.1). Later, Ohta \cite{Ohta0} obtained a weighted version of the Bonnet-Myers theorem on forward complete connected Finsler metric measure manifold $(M, F, m)$. He proved that, if the weighted Ricci curvature  ${\rm Ric} _{N} \geq (N-1)K$ for some $K>0$ and $N \in [n, \infty)$, then $M$ is in fact compact and has finite fundamental group, and the diameter of $M$ satisfies ${\rm diam} (M) \leq \frac{\pi}{\sqrt{K}}$ (also see \cite{Ohta1}, Corollary 9.17). Further, the first author and Shen \cite{ChSh} derived a new weighted version of the Bonnet-Myers theorem on forward complete connected Finsler metric measure manifold $(M, F, m)$ satisfying that ${\rm Ric}_{\infty} \geq K, \  |{\bf S}|\leq \delta$ for some $K>0$ and $\delta \geq 0$. In this case, the diameter of $M$ satisfies the following
\[
 {\rm diam} (M) \leq \frac{\pi}{\sqrt{K}} \Big ( \frac{\delta}{\sqrt{K}} + \sqrt{ \frac{\delta^2}{K} + n-1}\Big).
\]
In particular, $M$ is compact (\cite{ChSh}, Theorem 4.1). Here, ${\rm Ric}_{\infty}$ is the $\infty$-weighted Ricci curvature and ${\bf S}$ denotes the S-curvature of  $(M, F, m)$. On the other hand, Zhao \cite{Zhao} studied the integral of the Ricci curvature over metric balls in a Finsler manifold and obtained a  theorem of Bonnet-Myers type on forward complete Finsler metric measure manifold $(M, F, m)$ with integral Ricci curvature bounds and satisfying ${\rm Ric}\geq -(n-1)k^{2}$ and $\Lambda_{F}\leq \delta^2$ for some $k\in \mathbb{R}$ and $\delta \geq 1$ (\cite{Zhao}, Theorem 1.1). Here, $\Lambda_{F}$ is the reversibility of $F$.

In this paper, we will mainly focus on the Bonnet-Myers theorem on Finsler metric measure manifolds with integral weighted Ricci curvature bounds. Firstly, we will introduce two Laplacian comparison theorems in Section \ref{Laplacian}. Then,  in Section \ref{volcom}, we will establish two relative volume comparison theorems  (Theorem \ref{Svol} and Theorem \ref{vol}) and a Bishop-Gromov volume comparison  theorem for nonconcentric balls (Theorem \ref{noncon}), all of which are controlled by the integral weighted Ricci curvature. We will also give a Bishop-Gromov volume comparison theorem under integral weighted Ricci curvature bounds (Theorem \ref{BG}) in Section \ref{volcom}. Finally, we will prove a theorem (Theorem \ref{boma}) of Bonnet-Myers type on Finsler metric measure manifolds with integral weighted Ricci curvature bounds  in Section \ref{bon-May}. Our results in this paper are all new in Finsler geometry.

\section{Preliminaries}

In this section, we briefly review some necessary definitions, notations and  fundamental results in Finsler geometry. For more details, we refer to \cite{BaoChernShen, ChernShen, Ohta1, Shen1}.

Let $M$ be an $n$-dimensional smooth manifold. A Finsler metric on manifold $M$ is a function $F: T M \longrightarrow[0, \infty)$  satisfying the following properties: (1) $F$ is $C^{\infty}$ on $TM\backslash\{0\}$; (2) $F(x,\lambda y)=\lambda F(x,y)$ for any $(x,y)\in TM$ and all $\lambda >0$; (3)  $F$ is strongly convex, that is, the matrix $\left(g_{ij}(x,y)\right)=\left(\frac{1}{2}(F^{2})_{y^{i}y^{j}}\right)$ is positive definite for any nonzero $y\in T_{x}M$. The pair $(M,F)$ is called a Finsler manifold and $g:=g_{ij}(x,y)dx^{i}\otimes dx^{j}$ is called the fundamental tensor of $F$.
For a non-vanishing smooth vector field $V$ on $M$, one can introduce the weighted Riemannian metric $g_V$ on $M$ given by
\be
g_{V}(y, w)=g_{ij}(x, V_x)y^i w^j  \label{weiRiem}
\ee
for $y,\, w\in T_{x}M$. In particular, $g_V(V,V)=F^{2}(x,V)$.

We define the reverse metric $\overleftarrow{F} $ of $F$ by $\overleftarrow{F}(x, y):=F(x,-y)$ for all $(x, y) \in T M$. It is easy to see that $\overleftarrow{F}$ is also a Finsler metric on $M$. A Finsler metric $F$ on $M$ is said to be reversible if $\overleftarrow{F}(x, y)=F(x, y)$ for all $(x, y) \in T M$. Otherwise, we say $F$ is irreversible.
In order to overcome the deficiencies  caused by irreversibility, Rademacher defined the reversibility $\Lambda_{F}$ of $F$ by
\[
\Lambda_{F}:=\sup _{(x, y) \in TM \backslash\{0\}} \frac{\overleftarrow{F}(x, y)}{F(x, y)}.
\]
Obviously, $\Lambda_{F} \in [1, \infty]$ and $\Lambda_{F}=1$ if and only if $F$ is reversible \cite{Ra}.

Let $(M,F)$ be a Finsler manifold of dimension $n$. The pull-back $\pi ^{*}TM$ admits a unique linear connection  $D$, which is called the Chern connection.  Given a non-vanishing vector field $V$ on $M$, the Riemannian curvature $R^V$ is defined by
$$
R^V(X, Y) Z=D_X^V D_Y^V Z-D_Y^V D_X^V Z-D_{[X, Y]}^V Z
$$
for any vector fields $X$, $Y$, $Z$ on $M$, where $D^{V}_{X}Y$ is the covariant derivative with respect to the reference vector $V$.  For two linearly independent vectors $V, W \in T_x M \backslash\{0\}$, the flag curvature is defined by
$$
\mathcal{K}^V(V, W)=\frac{g_V\left(R^V(V, W) W, V\right)}{g_V(V, V) g_V(W, W)-g_V(V, W)^2}.
$$
Then the Ricci curvature is defined as
$$
\operatorname{Ric}(V):=F(x, V)^{2} \sum_{i=1}^{n-1} \mathcal{K}^V\left(V, e_i\right),
$$
where $e_1, \ldots, e_{n-1}, \frac{V}{F(V)}$ form an orthonormal basis of $T_x M$ with respect to $g_V$.

For $x_1, x_2 \in M$, the distance from $x_1$ to $x_2$ is defined by
\be
d_{F}(x_{1}, x_{2}):=\inf _{\gamma} \int_{0}^{1} F(\gamma (t), \dot{\gamma}(t)) d t,
\ee
where the infimum is taken over all $C^1$ curves $\gamma:[0,1] \rightarrow M$ such that $\gamma(0)=$ $x_1$ and $\gamma(1)=x_2$. Note that $d_{F} \left(x_1, x_2\right) \neq d_{F} \left(x_2, x_1\right)$ unless $F$ is reversible.  The diameter of $M$ is defined by
\be
{\rm diam}(M):=\sup _{x_{1}, x_{2} \in M} \{d_{F}(x_{1}, x_{2})\}.
\ee

Now we define the forward and backward geodesic balls of radius $R$ with center at $x_0\in M$ by
$$
B^{+}_{R}(x_0):=\{z\in M \ |\ d_{F}(x_0,z)<R\},\ \ \ B^{-}_{R}(x_0):=\{z \in M \ | \ d_{F}(z,x_0)<R\}.
$$
Sometimes, we will denote $B_{R}:=B^{+}_{R}(x_0)$ for some $x_{0}\in M$ for simplicity. Further, let $S_{R}(x_{0})$ be the forward geodesic sphere of radius $R$ at the center $x_{0}$.

A $C^{\infty}$-curve $\gamma:[0,1] \rightarrow M$ is called a geodesic  if $F(\gamma, \dot{\gamma})$ is constant and it is locally minimizing. The exponential map $\exp_x: T_x M \rightarrow M$ is defined by $\exp_x(v)=\gamma(1)$ for $v \in T_x M$ if there is a geodesic $\gamma:[0,1] \rightarrow M$ with $\gamma(0)=x$ and $\dot{\gamma}(0)=v$. A Finsler manifold $(M, F)$ is said to be forward complete (resp. backward complete) if each geodesic defined on $[0, \ell)$ (resp. $(-\ell, 0])$ can be extended to a geodesic defined on $[0, \infty)$ (resp. $(-\infty, 0])$. We say $(M, F)$ is complete if it is both forward complete and backward complete. By Hopf-Rinow theorem on forward complete Finsler manifolds, any two points in $M$ can be connected by a minimal forward geodesic and the forward closed balls $\overline{B_R^{+}(p)}$ are compact (see \cite{BaoChernShen, Shen1}).

\vskip 2mm

Let $(M, F, m)$ be an $n$-dimensional Finsler metric measure manifold with a smooth measure $m$. Write the volume form $dm$ of $m$ as $d m = \sigma(x) dx^{1} dx^{2} \cdots d x^{n}$. Define
\be\label{Dis}
\tau (x, y):=\ln \frac{\sqrt{{\rm det}\left(g_{i j}(x, y)\right)}}{\sigma(x)}.
\ee
We call $\tau$ the distortion of $F$. It is natural to study the rate of change of the distortion along geodesics. For a vector $y \in T_{x} M \backslash\{0\}$, let $\sigma=\sigma(t)$ be the geodesic with $\sigma(0)=x$ and $\dot{\sigma}(0)=y$.  Set
\be\label{S}
{\bf S}(x, y):= \frac{d}{d t}\left[\tau(\sigma(t), \dot{\sigma}(t))\right]|_{t=0}.
\ee
$\mathbf{S}$ is called the S-curvature of $F$ \cite{ChernShen, shen}.  We say that $\mathbf{S}\geq c$ for some $c\in \mathbb{R}$ if $\mathbf{S}(x,y)\geq c F(x, y)$ for all $x\in M$ and $y \in T_{x} M \backslash\{0\}$. In particular, let
\be
\vartheta_{0} := \sup\limits_{(x,y)\in TM\setminus \{0\}}\frac{|{\bf S}(x,y)|}{F(x,y)}. \label{supS}
\ee
Obviously, $\vartheta_{0}$ is finite when $M$ is compact. In this paper, we always assume that $\vartheta_{0}$ is finite.

Let $Y$ be a $C^{\infty}$ geodesic field on an open subset $U \subset M$ and $\hat{g}=g_{Y}.$  Let
\be
d m:=e^{-\psi} {\rm Vol}_{\hat{g}}, \ \ \ {\rm Vol}_{\hat{g}}= \sqrt{{det}\left(g_{i j}\left(x, Y_{x}\right)\right)}dx^{1} \cdots dx^{n}. \label{voldecom}
\ee
It is easy to see that $\psi$ is given by
\be
\psi(x)= \ln \frac{\sqrt{\operatorname{det}\left(g_{i j}\left(x, Y_{x}\right)\right)}}{\sigma(x)} =\tau\left(x, Y_{x}\right), \label{Psi}
\ee
which is just the distortion of $F$ along $Y_{x}$ at $x\in M$ \cite{ChernShen, Shen1}. Let $y := Y_{x}\in T_{x}M$ (that is, $Y$ is a geodesic extension of $y\in T_{x}M$). Then, by the definitions of the S-curvature, we have
\beqn
&&  {\bf S}(x, y)= Y[\tau(x, Y)]|_{x} = d \psi (y),  \\
&&  \dot{\bf S}(x, y)= Y[{\bf S}(x, Y)]|_{x} =y[Y(\psi)],
\eeqn
where $\dot{\bf S}(x, y):={\bf S}_{|m}(x, y)y^{m}$ and ``$|$" denotes the horizontal covariant derivative with respect to the Chern connection \cite{shen, Shen1}. Further, the weighted Ricci curvatures are defined as follows \cite{ChSh, Ohta1}
\beq
{\rm Ric}_{N}(y)&=& {\rm Ric}(y)+ \dot{\bf S}(x, y) -\frac{{\bf S}(x, y)^{2}}{N-n},   \label{weRicci3}\\
{\rm Ric}_{\infty}(y)&=& {\rm Ric}(y)+ \dot{\bf S}(x, y). \label{weRicciinf}
\eeq
We say that ${\rm Ric}_{N}\geq K$ for some $K\in \mathbb{R}$ if ${\rm Ric}_{N}(v)\geq K F^{2}(x, v)$ for all $x\in M$ and $v\in T_{x}M$, where $N\in \mathbb{R}\setminus \{n\}$ or $N= \infty$.

\vskip 2mm

For a point $x \in M$ and a unit vector $y \in S_{x} M:=\left\{v \in T_{x} M \mid F(x, v)=1\right\}$, the cut value $i_{y}$ of $y$ is defined by
$$
i_{y}:=\sup \left\{t>0 \mid d_{F}\left(x, \exp_{x}(ty)\right)=t\right\}.
$$
Further, we define the injectivity radius $i_x$ at $x$ by $i_{x}:=\inf\limits _{y \in S_{x} M} i_{y}$. Let
$$
\mathcal{D}_{x}:=\left\{\exp _{x}(t y) \mid 0 \leq t<i_{y}, \ y \in S_{x} M\right\} \subset M
$$
and
$$
{\rm Cut}_{x}:=M-\mathcal{D}_{x}.
$$
${\rm Cut}_x$ and $\mathcal{D}_x$ are called the cut-locus and the cut-domain of $x$ respectively. Let
$$
\mathfrak{D}_{x}:=\left\{t y \mid 0 \leq t<i_{y}, \ y \in S_{x} M\right\} \subset T_{x} M .
$$
$\mathfrak{D}_x$ is called the tangent cut-domain at $x$. The exponential map
$$
\exp_x: \mathfrak{D}_{x} \rightarrow \mathcal{D}_x
$$
is an onto diffeomorphism. The cut-locus ${\rm Cut}_{x}$ of $x$ always has null measure and $d_{x}:= d_{F}(x, \cdot)$ is $C^1$ outside the cut-locus of $x$ (see \cite{BaoChernShen,Shen1}).

\vskip 2mm

For any Finsler metric $F$ on $M$, its dual metric
\be
F^{*}(x, \xi):=\sup\limits_{y\in T_{x}M\setminus \{0\}} \frac{\xi (y)}{F(x,y)}, \ \ \forall\ \xi \in T^{*}_{x}M \label{co-Finsler}
\ee
is a Finsler co-metric on $M$.  Naturally, we define the Legendre transformation ${\cal L}: TM \rightarrow T^{*}M$  on Finsler manifold $(M,F)$ by
$$
{\cal L}(y):=\left\{
\begin{array}{ll}
g_{y}(y, \cdot), & y\neq 0, \\
0, & y=0.
\end{array} \right.
$$
${\cal L}$ is a norm-preserving transformation, that is,  $F(x,y)=F^{*}(x, {\cal L}(y))$.  Let
$$
g^{*kl}(x,\xi):=\frac{1}{2}\left[F^{*2}\right]_{\xi _{k}\xi_{l}}(x,\xi).
$$
Then, for any $\xi ={\cal L}(y)$, we have
\be
g^{*kl}(x,\xi)=g^{kl}(x,y), \label{Fdual}
\ee
where $\left(g^{kl}(x,y)\right)= \left(g_{kl}(x,y)\right)^{-1}$.

Given a smooth function $u$ on $M$, we define the gradient vector $\nabla u(x)$ of $u$ at $x \in M$ by $\nabla u(x):=\mathcal{L}^{-1}(d u(x)) \in T_x M$. In a local coordinate system, we can express $\nabla u$ as
\be \label{nabna}
\nabla u(x)= \begin{cases}g^{* i j}(x, d u) \frac{\partial u}{\partial x^i} \frac{\partial}{\partial x^j}, & x \in M_u, \\ 0, & x \in M \backslash M_u,\end{cases}
\ee
where $M_{u}:=\{x \in M \mid d u(x) \neq 0\}$ \cite{Shen1}. In general, $\nabla u$ is only continuous on $M$, but smooth on $M_{u}$.
The Hessian of $u$ is defined by using Chern connection as
$$
\nabla^2 u(X, Y)=g_{\nabla u}\left(D_X^{\nabla u} \nabla u, Y\right).
$$
One can show that $\nabla^2 u(X, Y)$ is symmetric \cite{Ohta3, WuXin}.

Let $W^{1, p}(M)(p \geq 1)$ be the space of functions $u \in L^p(M)$ with $\int_M[F(\nabla u)]^p d m+\int_M[\overleftarrow{F}(\overleftarrow{\nabla} u)]^p d m<\infty$ and $W_0^{1, p}(M)$ be the closure of $\mathcal{C}_0^{\infty}(M)$ under the (absolutely homogeneous) norm
\be
\|u\|_{W^{1, p}(M)}:=\|u\|_{L^p(M)}+\frac{1}{2}\|F(\nabla u)\|_{L^p(M)}+\frac{1}{2}\|\overleftarrow{F}(\overleftarrow{\nabla} u)\|_{L^p(M)},
\ee
where $\mathcal{C}_0^{\infty}(M)$ denotes the set of all smooth compactly supported functions on M and $\overleftarrow{\nabla} u$ is the gradient vector of $u$ with respect to the reverse metric $\overleftarrow{F}$. In fact, $\overleftarrow{F}(\overleftarrow{\nabla} u)=F(\nabla(-u))$.

Given a differentiable vector field $V$ on $M$, its  divergence  $\operatorname{div}_{m} V$  is defined in the weak form by following divergence formula
$$
\int_M \phi \operatorname{div}_m V d m=-\int_M d \phi(V) d m
$$
for all $\phi \in \mathcal{C}_0^{\infty}(M)$. Further, the Finsler Laplacian $\Delta u$ is defined by
\be
\Delta u:=\operatorname{div}_m(\nabla u). \label{Lap}
\ee
From (\ref{Lap}), Finsler Laplacian is a nonlinear elliptic differential operator of the second order.
Further, noticing that $\nabla u$ is weakly differentiable, the Finsler Laplacian should be understood in the weak sense, that is, for $u \in W^{1,2}(M)$, $\Delta u$ is defined by
\be
\int_M \phi \Delta u d m:=-\int_M d \phi(\nabla u) dm  \label{Lap1}
\ee
for $\phi \in \mathcal{C}_0^{\infty}(M)$ \cite{Ohta1,Shen1}.

Given a weakly differentiable function $u$ and a vector field $V$ which does not vanish on $M_u$, the weighted Laplacian of $u$ on the weighted Riemannian manifold $\left(M, g_V, m\right)$ is defined by
\be
\Delta^{V} u:= {\rm div}_{m}\left(\nabla^V u\right),
\ee
where
$$
\nabla^V u:= \begin{cases}g^{ij}(x, V) \frac{\partial u}{\partial x^i} \frac{\partial}{\partial x^j} & \text { for } x \in M_u, \\ 0 & \text { for } x \notin M_u .\end{cases}
$$
Similarly, the weighted Laplacian should be understood in the weak sense. We note that $\nabla^{\nabla u}u=\nabla u$ and $\Delta^{\nabla u} u=$ $\Delta u$. Moreover, it is easy to see that $\Delta u= {\rm tr}_{\nabla u} \nabla^2 u-{\bf S}(\nabla u)$ on $M_u$ \cite{Ohta1, WuXin}.

\section{Integral weighted Ricci curvature and Laplacian comparison theorems}\label{Laplacian}

Let $(M, F, m)$ be an $n$-dimensional Finsler manifold with a smooth measure $m$ and $x \in M$. Remember that $\mathfrak{D}_x=\left\{t y \mid 0 \leq t<i_{y}, \ y \in S_{x} M\right\}$ and $\mathcal{D}_{x}=\exp _{x}(\mathfrak{D}_{x})$. For any $z \in \mathcal{D}_x$, we can choose the geodesic polar coordinates $(r, \theta)$ centered at $x$ for $z$ such that $r(z)=F(x, v)$ and $\theta^\alpha(z)=\theta^\alpha\left(\frac{v}{F(v)}\right)$ for $1 \leq \alpha \leq n-1$, where $r(z)=d_{F}(x, z)$ and $v=\exp _x^{-1}(z) \in T_x M \backslash\{0\}$. It is well known that the distance function $r$ starting from $x \in M$ is smooth on $\mathcal{D}_x$ and $F(\nabla r)=1$ \cite{BaoChernShen,Shen1}. A basic fact is that the distance function $r=d_{F}(x, \cdot)$ satisfies the following \cite{BaoChernShen,Shen1,WuXin}
$$
\left.\nabla r\right|_z=\left.\frac{\partial}{\partial r}\right|_z .
$$

By Gauss's lemma, the unit radial coordinate vector $\frac{\partial}{\partial r}$ and the coordinate vectors $\frac{\partial}{\partial \theta^\alpha}$ for $1 \leq \alpha \leq n-1$ are mutually vertical with respect to $g_{\nabla r}$ (\cite{BaoChernShen}, Lemma 6.1.1). Therefore, we can simply write the volume form at $z=\exp _x(r \xi)$ with $v=r \xi$ as $\left.d m\right|_{\exp _x(r \xi)}=\sigma(x, r, \theta) d r d \theta$, where $\xi \in S_x M$. Then, for the forward geodesic ball $B_R=B_R^{+}(x)$ of radius $R$ at the center $x \in M$, the volume of $B_R$ is
$$
m\left(B_{R}\right)=\int_{B_{R}} d m=\int_{B_{R} \cap \mathcal{D}_{x}} d m=\int_{0}^{R} d r \int_{\mathcal{D}_x(r)} \sigma(x, r, \theta) d \theta = \int_{S_{x} M} \int_{0}^{\min \left\{R, i_{y}\right\}} \sigma(x, r, \theta) d r d \theta,
$$
where $\theta^{\alpha}=\theta^{\alpha}(y)$ for $y \in S_x M$ and $\mathcal{D}_x(r)=\left\{y \in S_x M \mid r y \in \mathfrak{D}_{x}\right\}$  \cite{Ohta1,Shen1}. Obviously, for any $0<s< t<R$, $\mathcal{D}_x(t) \subseteq \mathcal{D}_x(s)$. Besides, by the definition of Laplacian, we have \cite{Shen1,WuXin}
\be
\Delta r=\frac{\partial}{\partial r} \ln \sigma(x, r, \theta). \label{DelLa}
\ee

Let
$$
\underline{{\rm{Ric}}}_{\infty}(z):=\min _{y \in T_z M \backslash\{0\}} \frac{{\rm{Ric}}_{\infty}(z, y)}{F^2(z, y)}
$$
and ${\rm{Ric}}_{\infty}^K(z):=\max \left\{(n-1) K-\underline{{\rm{Ric}}}_{\infty}(z), 0\right\}$ for some $K \in \mathbb{R}$. Given $p \geq 1, \vartheta \geq 0$ and $R>0$, let
\[
\left\|{\rm{Ric}}_{\infty}^K\right\|_{p, R, \vartheta}(x):=\left(\int_0^R \int_{\mathcal{D}_x(r)}\left({\rm{Ric}}_{\infty}^K\right)^p \mathrm{e}^{-\vartheta r} \sigma(x, r, \theta) d r d \theta\right)^{\frac{1}{p}}.
\]
To remove the dependence on the volume of the geodesic ball, we further define
\be
\overline{\left\|{\rm{Ric}}_{\infty}^K\right\|}_{p, R, \vartheta}(x) :=\frac{\left\|{\rm{Ric}}_{\infty}^K\right\|_{p, R, \vartheta}(x)}{m\left(B_R^{+}(x)\right)^{\frac{1}{p}}}=\left(\frac{1}{m\left(B_R^{+}(x)\right)} \int_0^R \int_{\mathcal{D}_x(r)}\left({\rm{Ric}}_{\infty}^K\right)^p \mathrm{e}^{-\vartheta r} \sigma(x, r, \theta) d r d \theta\right)^{\frac{1}{p}} \label{iwRp}
\ee
and
\be
{\overline{\left\|{\rm{Ric}}_{\infty}^K\right\|}_{p, R, \vartheta}}:=\sup_{x \in M}{\overline{\left\|{\rm{Ric}}_{\infty}^K\right\|} _{p, R, \vartheta}}(x). \label{iwRM}
\ee
${\overline{\left\|{\rm{Ric}}_{\infty}^K\right\|}_{p, R, \vartheta}}$ measures how much the weighted Ricci curvature ${\rm{Ric}}_{\infty}$ lies below a given bound $(n-1) K$ in the $L^p$ sense. Clearly, $\overline{\left\|{\rm{Ric}}_{\infty}^K\right\|}_{p, R, \vartheta}=0$ iff ${\rm{Ric}}_{\infty} \geq (n-1) K$.

\begin{rem}\label{epsilon}
The norms $\left\|{\rm{Ric}}_{\infty}^K\right\|_{p, R, \vartheta}(x)$ and $\overline{\left\|{\rm{Ric}}_{\infty}^K\right\|}_{p, R, \vartheta}(x)$ are both non-decreasing in $K$. This follows directly from the fact that ${\rm{Ric}}_{\infty}^K(z)$ is non-decreasing in $K$.
\end{rem}

\begin{rem}
By (\ref{iwRp}) and (\ref{iwRM}), we can deduce that for $0<r\leq R$,
\beq\label{ric<}
\overline{\left\|{\rm{Ric}}_{\infty}^K\right\|}_{p, r, \vartheta} &\leq& \sup_{x\in M} \left(\frac{\left\|{\rm{Ric}}_{\infty}^K\right\|_{p, R, \vartheta}(x)} {m(B_{r}^{+}(x))^{\frac{1}{p}}}\right)
= \sup_{x\in M} \left(\frac{m(B_{R}^{+}(x))^{\frac{1}{p}}}{m(B_{r}^{+}(x))^{\frac{1}{p}}}\cdot\frac{\left\|{\rm{Ric}}_{\infty}^K\right\|_{p, R, \vartheta}(x)} {m(B_{R}^{+}(x))^{\frac{1}{p}}}\right)\nonumber\\
&\leq & \sup_{x\in M} \left(\frac{m(B_{R}^{+}(x))}{m(B_{r}^{+}(x))}\right)^{\frac{1}{p}} \overline{\left\|{\rm{Ric}}_{\infty}^K\right\|}_{p, R, \vartheta}.
\eeq
\end{rem}

Noticing that $\Delta r={\rm{tr}}_{\nabla r}\left(\nabla^2 r\right)-\mathbf{S}(\nabla r)$, let $h:={\rm{tr}}_{\nabla r}\left(\nabla^2 r\right)$ and $H_K(r):=(n-1) \frac{s_K^{\prime}(r)}{s_K(r)}$, where
$$
s_K(t):= \begin{cases}\frac{1}{\sqrt{K}} \sin (\sqrt{K} t), & K>0, \\ t, & K=0, \\ \frac{1}{\sqrt{-K}} \sinh (\sqrt{-K} t), & K<0.\end{cases}
$$
Let
$$
\varphi(r, \theta):= \begin{cases}\left(\Delta r-H_K(r)-\vartheta\right)_{+}, & 0 \leq r<i_x, \\ 0, & r \geq i_x,\end{cases}
$$
where $(\star)_{+}:=\max \{\star, 0\}$.
It is clear that, if $\Delta r \geq H_{K}+\vartheta$, then $\varphi=\Delta r-H_{K}(r)-\vartheta$. Otherwise, $\varphi=0$. Moreover, if $\mathbf{S}(\nabla r) \geq -\vartheta$, then
\be
\lim_{r \rightarrow 0^{+}} \varphi(r, \theta)=\lim _{r \rightarrow 0^{+}}\left(h-\mathbf{S}(\nabla r)-H_K(r)-\vartheta\right)_{+}=\lim _{r \rightarrow 0^{+}}(-\mathbf{S}(\nabla r)-\vartheta)_{+} = 0 \label{phia0}
\ee
because $h \sim \frac{n-1}{r}$ and $H_{K}(r) \sim \frac{n-1}{r}$ when $r \rightarrow 0^{+}$. Further,  by (5.1) in \cite{WuXin}, that is,
$$
\frac{d}{d r} {\rm{tr}}_{\nabla r}\left(\nabla^2 r\right)+\sum_{i, j}\left(\nabla^2 r\left(E_i, E_j\right)\right)^2=-{\rm{Ric}}(\nabla r),
$$
where $E_1, \cdots, E_{n-1}, E_n=\nabla r$ is the local $g_{\nabla r}$-orthonormal frame along the geodesic, Cheng-Feng \cite{ChF} obtained the following inequality (see (3.7) in \cite{ChF})
\beq
\frac{\partial}{\partial r} \varphi+\frac{\varphi^2}{n-1}+2 \frac{\varphi H_K}{n-1} \leq(n-1) K-{\rm{Ric}}_{\infty}(\nabla r) \leq {\rm{Ric}}_{\infty}^{K}. \label{phi}
\eeq
From (\ref{phi}), Cheng-Feng proved the following Laplacian comparison theorem in \cite{ChF}. It is important for the proof of our later theorems.

\begin{thm}{\rm(\cite{ChF}, Theorem 3.1)}\label{Lap1}
Let $(M, F, m)$ be an $n$-dimensional Finsler metric measure manifold. Assume that $\mathbf{S}\geq -\vartheta$ for some $\vartheta \geq 0$.
Then, for $p>\frac{n}{2}$, $K \in \mathbb{R}$ and $r>0$ $(r \leq \frac{\pi}{2 \sqrt{K}}$ when $K>0)$, we have
$$
\int_0^r \varphi^{2 p}(t, \theta) \mathrm{e}^{-\vartheta t} \sigma(t, \theta) d t \leq \left(\frac{(n-1)(2 p-1)}{2 p-n}\right)^p \int_0^r\left(\operatorname{Ric}_{\infty}^{K}\right)^{p} \mathrm{e}^{-\vartheta t} \sigma(t, \theta) d t
$$
and
\be
\varphi(r, \theta)^{2 p-1} \mathrm{e}^{-\vartheta r} \sigma(r, \theta) \leq (2 p-1)^p\left(\frac{n-1}{2 p-n}\right)^{p-1} \int_0^r\left(\operatorname{Ric}_{\infty}^K\right)^p \mathrm{e}^{-\vartheta t} \sigma(t, \theta) d t, \label{flapcom}
\ee
where $(t, \theta)$ denotes the geodesic polar coordinates centered at $x \in M$.
\end{thm}

Actually, by the similar argument, we can easily get the following Laplacian comparison theorem for the case that $K>0$ and $0 < r <\frac{\pi}{\sqrt{K}}$.

\begin{thm}\label{sin}
Let $(M, F, m)$ be an $n$-dimensional Finsler metric measure manifold. Assume that $\mathbf{S}\geq -\vartheta$ for some $\vartheta \geq 0$. Then for $p>\frac{n}{2}$, $K>0$ and $0 < r <\frac{\pi}{\sqrt{K}}$, we have
\be
\sin ^{4 p-n-1}(\sqrt{K} r) \varphi(r, \theta)^{2 p-1} e^{-\vartheta r} \sigma(r, \theta) \leq (2 p-1)^{p}\left(\frac{n-1}{2 p-n}\right)^{p-1} \int_{0}^{r}\left({\rm{Ric}}_{\infty}^{K}\right)^{p} e^{-\vartheta t} \sigma(t, \theta) d t. \label{flapcom1}
\ee
\end{thm}
\begin{proof}
The proof mainly follows the argument of Theorem 3.1 in \cite{ChF}. For the reader's convenience, we give a brief proof of the theorem. Let $\phi(t)=\sin^{4p-n-1}(\sqrt{K}t)$. Through a simple calculation, we have
\beq\label{phi0}
\frac{\partial \phi}{\partial t}=\frac{4p-n-1}{n-1}\phi H_{K}(t).
\eeq
Then,  from (\ref{DelLa}), (\ref{phi}) and (\ref{phi0}), we get
$$
\frac{\partial}{\partial t}\left(\phi\varphi^{2p-1}e^{-\vartheta t}\sigma\right)
\leq (2p-1)\phi \varphi^{2p-2} e^{-\vartheta t} \sigma {\rm{Ric}}^{K}_{\infty}-\left(\frac{2p-n}{n-1}\right)\phi\varphi^{2p}e^{-\vartheta t}\sigma.
$$
Integrating it with respect to $t$ from $0$ to $r$ and using (\ref{phia0}) yield
$$
\phi \varphi^{2p-1}e^{-\vartheta r} \sigma(r,\theta)\leq (2p-1)\int^{r}_{0}\phi\varphi^{2p-2} e^{-\vartheta t}  {\rm{Ric}}^{K}_{\infty} \cdot \sigma(t,\theta) dt - \left(\frac{2p-n}{n-1}\right)\int^{r}_{0} \phi\varphi^{2p}e^{-\vartheta t}\sigma(t,\theta) dt,
$$
which implies
\beq\label{phi1}
\phi \varphi^{2p-1} e^{-\vartheta r}\sigma(r,\theta) \leq (2p-1)\int^{r}_{0}\phi \varphi^{2p-2}e^{-\vartheta t} {\rm{Ric}}^{K}_{\infty} \cdot  \sigma(t,\theta) dt
\eeq
and
\beq\label{phi2}
\frac{2p-n}{n-1}\int^{r}_{0}\phi \varphi^{2p}e^{-\vartheta t}\sigma(t,\theta) dt \leq (2p-1)\int^{r}_{0}\phi \varphi^{2p-2} e^{-\vartheta t} {\rm{Ric}}^{K}_{\infty} \cdot  \sigma(t,\theta) dt.
\eeq
By H\"{o}lder inequality and (\ref{phi2}), we obtain the following
\beqn
&& \int_{0}^{r} \phi \varphi^{2 p-2} e^{-\vartheta t}  {\rm{Ric}}_{\infty}^{K} \cdot \sigma(t,\theta) dt
\leq \left[\int_{0}^{r} \phi \varphi^{2p} e^{-\vartheta t} \sigma(t,\theta) d t \right]^{\frac{p-1}{p}} \left[\int_{0}^{r} \phi e^{-\vartheta t} \left({\rm{Ric}}_{\infty}^{K}\right)^{p} \cdot \sigma(t,\theta) dt \right]^{\frac{1}{p}} \\
&& \leq \left[\frac{(2 p-1)(n-1)}{2 p-n}\int_{0}^{r} \phi \varphi^{2 p-2} e^{-\vartheta t}  {\rm{Ric}}_{\infty}^{K} \cdot \sigma(t,\theta) dt\right]^{\frac{p-1}{p}}\left[\int_{0}^{r} \phi e^{-\vartheta t} \left({\rm{Ric}}_{\infty}^{K}\right)^{p} \cdot \sigma(t,\theta)d t\right]^{\frac{1}{p}}.
\eeqn
Then we get
\beq
&& \int^{r}_{0}\phi \varphi^{2 p-2}(t, \theta) e^{-\vartheta t}  {\rm{Ric}}_{\infty}^{K} \cdot \sigma(t, \theta) dt \nonumber\\
&& \leq  \left(\frac{(2 p-1)(n-1)}{2 p-n}\right)^{p-1}\int^{r}_{0}\sin^{4p-n-1}(\sqrt{K}t) \left({\rm{Ric}}_{\infty}^{K}\right)^{p} e^{-\vartheta t}  \sigma(t,\theta) dt \nonumber\\
&& \leq  \left(\frac{(2 p-1)(n-1)}{2 p-n}\right)^{p-1}\int^{r}_{0}  \left({\rm{Ric}}_{\infty}^{K}\right)^{p} e^{-\vartheta t}  \sigma(t,\theta) dt.  \label{phi3}
\eeq
By substituting (\ref{phi3}) into (\ref{phi1}), we obtain
\[
\sin ^{4 p-n-1}(\sqrt{K} r) \varphi^{2 p-1}(r, \theta) e^{-\vartheta r} \sigma(r, \theta)\leq (2 p-1)^{p}\left(\frac{n-1}{2 p-n}\right)^{p-1} \int_{0}^{r} \left({\rm{Ric}}_{\infty}^{K}\right)^{p} e^{-\vartheta t} \sigma(t, \theta) d t.
\]
\end{proof}
\vskip 2mm

It should be pointed out that, when $K>0$ and $0< r \leq \frac{\pi}{2 \sqrt{K}}$, (\ref{flapcom}) is sharper than (\ref{flapcom1}).

\section{Volume comparison theorems}\label{volcom}

In this section, we will give some new volume comparison theorems on Finsler metric measure manifolds, which are important for our discussions in Section \ref{bon-May}.  For the sake of convenience, we set $A(r)=A(x_{0},r):=\int_{\mathcal{D}_{x_{0}}(r)}\sigma(r,\theta)d\theta$ for any given $x_{0}\in M$, where $(r, \theta)$ denotes the geodesic polar coordinates centered at $x_{0} \in M$. Actually, $A(r)= m\left(S_{r}(x_{0})\right)$. Thus the following theorem can be regarded as a relative volume comparison theorem.

\begin{thm}\label{Svol}
Let $(M, F, m)$ be an $n$-dimensional Finsler metric measure manifold. Assume that $\mathbf{S}\geq -\vartheta$ for some $\vartheta \geq 0$.
Then for $p>\frac{n}{2}$ and $K>0$, we have the following inequalities.
\ben
\item[{\rm (1)}] If $0< r\leq R\leq \frac{\pi}{2 \sqrt{K}}$, then
\beq
& &\left[\frac{A(R)}{e^{\vartheta R} \sin^{n-1}(\sqrt{K} R)}\right]^{\frac{1}{2p-1}} -\left[\frac{A(r)}{e^{\vartheta r} \sin^{n-1}(\sqrt{K} r)}\right]^{\frac{1}{2p-1}} \nonumber\\
&&  \leq \left(\frac{n-1}{(2p-1)(2p-n)}\right)^{\frac{p-1}{2 p-1}}\left(\left\|{\rm Ric}_{\infty}^{K}\right\|_{p, R, \vartheta}(x_{0})\right)^{\frac{p}{2p-1}} \int_{r}^{R}\left(\frac{1}{\sin (\sqrt{K} s)}\right)^{\frac{n-1}{2p-1}} ds. \label{revolcom1}
\eeq

\item[{\rm (2)}] If $\frac{\pi}{2 \sqrt{K}}<r \leq R <\frac{\pi}{\sqrt{K}}$, then
\beq
&& \left[\frac{A(R)}{e^{\vartheta R} \sin^{n-1}(\sqrt{K} R)}\right]^{\frac{1}{2p-1}}-\left[\frac{A(r)}{e^{\vartheta r} \sin^{n-1}(\sqrt{K} r)}\right]^{\frac{1}{2p-1}}\nonumber\\
&& \leq  \left(\frac{n-1}{(2 p-1)(2 p-n)}\right)^{\frac{p-1}{2p-1}}\left(\left\|{\rm Ric}_{\infty}^{K}\right\|_{p, R, \vartheta}(x_{0})\right)^{\frac{p}{2p-1}} \int_{r}^{R} \frac{1}{\sin^{2}(\sqrt{K} s)} ds. \label{area2}
\eeq
\een
\end{thm}
\begin{proof}
Let function $f(s):=\int_{\mathcal{D}_{x_{0}}(R)} \sigma(s, \theta) d \theta$. It is obvious that $f(s)$ is differentiable on $(0, R]$. By the definition of $H_{K}(r)$, when $K>0$, we have
$$
H_{K}(r)=(n-1) \sqrt{K} \frac{\cos (\sqrt{K} r)}{\sin (\sqrt{K} r)}.
$$
From (\ref{DelLa}), we have $\frac{\partial}{\partial r} \sigma(r, \theta)=\sigma(r, \theta) \Delta r$. Noticing that $\frac{\partial}{\partial r}\left(e^{\vartheta r} s_{K}^{n-1}(r)\right)=e^{\vartheta r} s_{K}^{n-1}(r)\left(H_{K}(r)+\vartheta\right)$, for any constant $\alpha >0$, we can get
\beqn
\frac{d}{d s}\left[\frac{f(s)}{e^{\vartheta s} s_{K}^{n-1}(s)}\right]^{\alpha}& =& \alpha\left[\frac{f(s)}{e^{\vartheta s} s_{K}^{n-1}(s)}\right]^{\alpha-1}\left[\frac{f^{\prime}(s)-f(s)\left(H_{K}(s)+\vartheta\right)}{e^{\vartheta s} s_{K}^{n-1}(s)}\right]\\
& =& \alpha\left[\frac{f(s)}{e^{\vartheta s} s_{K}^{n-1}(s)}\right]^{\alpha-1}\left[\frac{\int_{\mathcal{D}_{x_{0}}(R)}{\sigma(s, \theta)\left(\Delta s-H_{K}(s)-\vartheta\right) d \theta}}{e^{\vartheta s} s_{K}^{n-1}(s)}\right] \\
& \leq & \alpha \frac{(f(s))^{\alpha-1}}{\left(e^{\vartheta s} s_{K}^{n-1}(s)\right)^{\alpha}}\int_{\mathcal{D}_{x_{0}}(R)} \varphi(s, \theta) \sigma(s, \theta) d \theta.
\eeqn
For any $0<r \leq R$, we know that $\mathcal{D}_{x_{0}}(R) \subseteq \mathcal{D}_{x_{0}}(r)$. Integrating both sides of above inequality in $s$ from $r$ to $R$ yields
\beq
&& \left[\frac{A(R)}{e^{\vartheta R} s_{K}^{n-1}(R)}\right]^{\alpha}-\left[\frac{A(r)}{e^{\vartheta r} s_{K}^{n-1}(r)}\right]^{\alpha}
\leq  \left[\frac{f(R)}{e^{\vartheta R} s_{K}^{n-1}(R)}\right]^{\alpha}-\left[\frac{f(r)}{e^{\vartheta r} s_{K}^{n-1}(r)}\right]^{\alpha}\nonumber \\
&& =  \int_{r}^{R} \frac{d}{d s}\left[\frac{f(s)}{e^{\vartheta s} s_{K}^{n-1}(s)}\right]^{\alpha} d s \leq \alpha \int_{r}^{R} \frac{(f(s))^{\alpha-1}}{\left(e^{\vartheta s} s_{K}^{n-1}(s)\right)^{\alpha}} \int_{\mathcal{D}_{x_{0}}(R)}  \varphi(s, \theta)  \sigma(s, \theta) d\theta ds,\label{Af1}
\eeq
where we have used the facts that $A(R)=f(R)$ and $A(r)\geq f(r)$. Furthermore, by H\"{o}lder's inequality, we get
\beq
\int_{\mathcal{D}_{x_{0}}(R)}  \varphi(s, \theta) \sigma(s, \theta) d\theta &\leq & \left(\int_{\mathcal{D}_{x_{0}}(R)}\sigma(s,\theta) d\theta\right)^{\frac{2p-2}{2p-1}}\left(\int_{\mathcal{D}_{x_{0}}(R)}  \varphi^{2p-1}(s, \theta) \sigma(s, \theta) d\theta\right)^{\frac{1}{2 p-1}}\nonumber \\
& = & [f(s)]^{\frac{2 p-2}{2 p-1}}\left(\int_{\mathcal{D}_{x_{0}}(R)}  \varphi^{2p-1}(s, \theta) \sigma(s, \theta) d\theta\right)^{\frac{1}{2 p-1}}. \label{Af2}
\eeq
Substituting (\ref{Af2}) into (\ref{Af1}) and taking $\alpha=\frac{1}{2 p-1}$, when $K>0$, we obtain
\beq
&& \left[\frac{A(R)}{e^{\vartheta R} \sin^{n-1} (\sqrt{K} R)}\right]^{\frac{1}{2 p-1}}-\left[\frac{A(r)}{e^{\vartheta r} \sin^{n-1} (\sqrt{K} r)}\right]^{\frac{1}{2 p-1}}\nonumber\\
&& \leq \frac{1}{2 p-1} \int_{r}^{R} \left(\frac{1}{e^{\vartheta s} \sin^{n-1} (\sqrt{K} s)} \int_{\mathcal{D}_{x_{0}}(R)}  \varphi^{2p-1}(s, \theta) \sigma(s, \theta) d\theta\right)^{\frac{1}{2 p-1}} ds. \label{Af3}
\eeq

(1) When $0<r \leq R\leq \frac{\pi}{2 \sqrt{K}}$, combining (\ref{Af3}) with  (\ref{flapcom}), we can get
\beqn
&& \left[\frac{A(R)}{e^{\vartheta R} \sin^{n-1} (\sqrt{K} R)}\right]^{\frac{1}{2 p-1}}-\left[\frac{A(r)}{e^{\vartheta r} \sin^{n-1} (\sqrt{K} r)}\right]^{\frac{1}{2 p-1}}\\
&& \leq \frac{1}{2 p-1} \int_{r}^{R} \frac{1}{\sin^{\frac{n-1}{2 p-1}} (\sqrt{K} s)} \left(\int_{\mathcal{D}_{x_{0}}(R)} e^{-\vartheta s}  \varphi^{2p-1}(s, \theta) \sigma(s, \theta) d\theta\right)^{\frac{1}{2 p-1}} d s\\
&& \leq  \left(\frac{n-1}{(2p-1)(2p-n)}\right)^{\frac{p-1}{2p-1}} \int_{r}^{R} \frac{1}{\sin^{\frac{n-1}{2 p-1}} (\sqrt{K} s)} \left(\int_{\mathcal{D}_{x_{0}}(R)} \int_{0}^{s}\left({\rm Ric}_{\infty}^{K}\right)^{p} e^{-\vartheta t} \sigma(t, \theta) d t d\theta\right)^{\frac{1}{2 p-1}} d s\\
&& \leq \left(\frac{n-1}{(2p-1)(2p-n)}\right)^{\frac{p-1}{2p-1}}\left\|{\rm{Ric}}_{\infty}^{K}\right\|_{p, R, \vartheta}^{\frac{p}{2p-1}}(x_{0})\int_{r}^{R} \left(\frac{1}{\sin(\sqrt{K} s)}\right)^{\frac{n-1}{2p-1}} d s,
\eeqn
which is just (\ref{revolcom1}).

(2) When $\frac{\pi}{2 \sqrt{K}}<r \leq R<\frac{\pi}{\sqrt{K}}$, combining (\ref{Af3}) with  (\ref{flapcom1}), we have
\beqn
&& \left[\frac{A(R)}{e^{\vartheta R} \sin^{n-1} (\sqrt{K} R)}\right]^{\frac{1}{2 p-1}}-\left[\frac{A(r)}{e^{\vartheta r} \sin^{n-1} (\sqrt{K} r)}\right]^{\frac{1}{2 p-1}}\\
&& \leq \frac{1}{2 p-1} \int_{r}^{R} \frac{1}{\sin^{2} (\sqrt{K} s)} \left(\int_{\mathcal{D}_{x_{0}}(R)}\sin^{4p-n-1} (\sqrt{K} s) e^{-\vartheta s}  \varphi^{2p-1}(s, \theta) \sigma(s, \theta) d\theta\right)^{\frac{1}{2 p-1}} d s\\
&& \leq \left(\frac{n-1}{(2p-1)(2p-n)}\right)^{\frac{p-1}{2p-1}} \int_{r}^{R} \frac{1}{\sin^{2} (\sqrt{K} s)} \left(\int_{\mathcal{D}_{x_{0}}(R)} \int_{0}^{s}\left({\rm{Ric}}_{\infty}^{K}\right)^{p} e^{-\vartheta t} \sigma(t, \theta) d t d\theta\right)^{\frac{1}{2 p-1}} d s\\
&& \leq \left(\frac{n-1}{(2p-1)(2p-n)}\right)^{\frac{p-1}{2p-1}}\left\|{\rm{Ric}}_{\infty}^{K}\right\|_{p, R, \vartheta}^{\frac{p}{2p-1}}(x_{0})\int_{r}^{R} \frac{1}{\sin^{2}(\sqrt{K} s)} d s.
\eeqn
This completes the proof.
\end{proof}
\vskip 2mm

As an application of Theorem \ref{Svol}, we derive an upper bound for $A(r)$ under the integral weighted Ricci curvature bound.

\begin{lem}\label{S(r)}
Let $(M, F, m)$ be an $n$-dimensional $(n \geq 2)$ Finsler metric measure manifold. Assume that $\mathbf{S}\geq -\vartheta$ for some $\vartheta \geq 0$.
Further, for $p>\frac{n}{2}$ and  $R_{0}>0$, assume that $\overline{\left\|{\rm{Ric}}_{\infty}^{\Lambda_{F}^2}\right\|}_{p, R_{0}, \vartheta}(x_{0})\leq \left(\frac{\pi}{6}\right)^{2-\frac{1}{p}}$. Then, for any $r \in \left[\frac{\pi}{\Lambda_{F}}, \frac{R_0}{\Lambda_{F}}\right]$, there exists a positive constant $C_{1}=C_{1}(n,p,\Lambda_{F},R_{0})$, such that
\beq
A(r) &\leq &  C_{1} e^{\vartheta r} \overline{\left\|{\rm{Ric}}_{\infty}^{\Lambda_{F}^2}\right\|}_{p, R_{0}, \vartheta}^{\frac{(n-1)p}{2p-1}}(x_{0}) m\left(B_{R_{0}}^{+}(x_{0})\right)\nonumber\\
&\leq &  C_{1} e^{\vartheta r} \left(\frac{\pi}{6}\right)^{n-1} m\left(B_{R_{0}}^{+}(x_{0})\right), \label{lem4.2}
\eeq
where
$
C_{1}(n,p,\Lambda_{F},R_{0}):=\frac{3\cdot 2^{2p+n-3}\Lambda_{F}}{\pi} +\frac{\pi^{2p-1}}{4^{p}}\left(\frac{n-1}{(2p-1)(2 p-n)}\right)^{p-1} \Lambda_{F}^{-(2p-1)}R_{0}^{2p-1}.
$
\end{lem}
\begin{proof}
For simplicity, let $\epsilon:=\overline{\left\|{\rm{Ric}}_{\infty}^{\Lambda_{F}^2}\right\|}_{p, R_{0}, \vartheta}^{\frac{p}{2 p-1}}(x_{0})$ and $K_{r}:=\left(\frac{\pi-\epsilon}{r}\right)^{2}$ for any $r \in \left[\frac{\pi}{\Lambda_{F}}, \frac{R_0}{\Lambda_{F}}\right]$. By the assumption, $\epsilon \leq \frac{\pi}{6}$, $\frac{5\pi}{6\sqrt{K_{r}}}\leq r \leq \frac{\pi}{\sqrt{K_{r}}}$ and $K_{r} \leq \Lambda_{F}^2$.
For $t \in\left[\frac{\pi}{2\sqrt{K_{r}}}, r\right]$, we have $\sqrt{K_{r}} t \in\left[\frac{\pi}{2}, \pi\right)$ and
\beqn
\int_{t}^{r} \frac{1}{\sin^{2}\left(\sqrt{K_{r}} s\right)} ds
&\leq &  \left(\frac{\pi}{2}\right)^{2} \int_{t}^{r} \frac{1}{\left(\pi-\sqrt{K_{r}}s\right)^{2}} d s = \frac{\pi^{2}}{4}\frac{(r-t)}{(\pi-\sqrt{K_r}t)(\pi-\sqrt{K_r}r)}\\
&= & \frac{\pi^{2}}{4}\frac{(r-t)}{(\pi-\sqrt{K_r}t)\epsilon}\leq \frac{\pi r}{4 \epsilon},
\eeqn
where we have used the facts that $\sin \theta\geq \frac{2}{\pi}(\pi-\theta)$ for $\theta\in (\frac{\pi}{2}, \pi)$ and $t\rightarrow \frac{r-t}{\pi-\sqrt{K_r}t}$ is decreasing on $[0, \frac{\pi}{\sqrt{K_{r}}})$. By Remark \ref{epsilon}, we know that $\left\|{\rm Ric}_{\infty}^{K_r}\right\|_{p, r, \vartheta}^{\frac{p}{2 p-1}}(x_{0}) \leq \left\|{\rm{Ric}}_{\infty}^{\Lambda_{F}^2}\right\|_{p, r, \vartheta}^{\frac{p}{2 p-1}}(x_{0})$. Because  $\frac{\pi}{2\sqrt{K_{r}}}< t < r<\frac{\pi}{\sqrt{K_{r}}}$,  it follows from (\ref{area2}) with $K=K_{r}$ that
\beqn
&&  \left[\frac{A(r)}{e^{\vartheta r} \sin^{n-1}\left(\sqrt{K_{r}} r\right)}\right]^{\frac{1}{2 p-1}}-\left[\frac{A(t)}{e^{\vartheta t} \sin^{n-1}\left(\sqrt{K_{r}} t\right)}\right]^{\frac{1}{2 p-1}} \\
&& \leq \left(\frac{n-1}{(2p-1)(2p-n)}\right)^{\frac{p-1}{2p-1}} \left\|{\rm{Ric}}_{\infty}^{K_r}\right\|_{p, r, \vartheta}^{\frac{p}{2 p-1}}(x_{0}) \int_{t}^{r} \frac{1}{\sin^{2}\left(\sqrt{K_{r}} s\right)} ds\\
&& \leq \left(\frac{n-1}{(2p-1)(2p-n)}\right)^{\frac{p-1}{2p-1}} \left\|{\rm{Ric}}_{\infty}^{\Lambda_{F}^2}\right\|_{p, r, \vartheta}^{\frac{p}{2p-1}}(x_{0}) \frac{\pi r}{4 \epsilon}.
\eeqn
From the above inequality and by the fact that $\sin \left(\sqrt{K_{r}} r\right) = \sin(\pi-\epsilon)=\sin \epsilon\leq \epsilon$, we get
\beq\label{A(r)}
\left(A(r)\right)^{\frac{1}{2p-1}} &\leq & A(t)^{\frac{1}{2 p-1}}\left(e^{\vartheta (r-t)}\right)^{\frac{1}{2 p-1}}\left(\frac{\epsilon}{\sin \left(\sqrt{K_{r}}t\right)}\right)^{\frac{n-1}{2p-1}} \nonumber\\
&& +\left(\frac{n-1}{(2 p-1)(2 p-n)}\right)^{\frac{p-1}{2 p-1}} \left(e^{\vartheta r}\right)^{\frac{1}{2p-1}}  \left\|{\rm{Ric}}_{\infty}^{\Lambda_{F}^2}\right\|_{p, r, \vartheta}^{\frac{p}{2p-1}}(x_{0}) \frac{\pi r}{4} \epsilon^{\frac{n-2p}{2 p-1}} \nonumber\\
&\leq & A(t)^{\frac{1}{2p-1}}\left(e^{\vartheta (r-t)}\right)^{\frac{1}{2 p-1}}\left(\frac{\epsilon}{\sin \left(\sqrt{K_{r}} t\right)}\right)^{\frac{n-1}{2p-1}} \nonumber\\
&& +\left(\frac{n-1}{(2p-1)(2p-n)}\right)^{\frac{p-1}{2p-1}} \left(e^{\vartheta r}\right)^{\frac{1}{2p-1}} \left\|{\rm{Ric}}_{\infty}^{\Lambda_{F}^2}\right\|_{p, R_{0}, \vartheta}^{\frac{p}{2p-1}}(x_{0}) \frac{\pi r}{4} \epsilon^{\frac{n-2p}{2 p-1}} \nonumber\\
&\leq & A(t)^{\frac{1}{2p-1}}\left(e^{\vartheta r}\right)^{\frac{1}{2 p-1}}\left(\frac{\epsilon}{\sin \left(\sqrt{K_{r}} t\right)}\right)^{\frac{n-1}{2p-1}} \nonumber\\
&& +\left(\frac{n-1}{(2p-1)(2p-n)}\right)^{\frac{p-1}{2p-1}}  \left(e^{\vartheta r}\right)^{\frac{1}{2p-1}} m\left(B_{R_{0}}^{+}(x_{0})\right)^{\frac{1}{2p-1}} \overline{\left\|{\rm{Ric}}_{\infty}^{\Lambda_{F}^2}\right\|}_{p, R_{0}, \vartheta}^{\frac{p}{2p-1}}(x_{0}) \frac{\pi r}{4} \epsilon^{\frac{n-2p}{2 p-1}} \nonumber\\
&= & A(t)^{\frac{1}{2p-1}}\left(e^{\vartheta r}\right)^{\frac{1}{2 p-1}}\left(\frac{\epsilon}{\sin \left(\sqrt{K_{r}} t\right)}\right)^{\frac{n-1}{2p-1}} \nonumber\\
&& +\left(\frac{n-1}{(2p-1)(2p-n)}\right)^{\frac{p-1}{2p-1}}  \left(e^{\vartheta r}\right)^{\frac{1}{2p-1}} m\left(B_{R_{0}}^{+}(x_{0})\right)^{\frac{1}{2p-1}} \frac{\pi r}{4} \epsilon^{\frac{n-1}{2 p-1}},
\eeq
where we used the fact that $\left\|{\rm{Ric}}_{\infty}^{\Lambda_{F}^2}\right\|_{p, r, \vartheta}^{\frac{p}{2p-1}}(x_{0})$ is increasing with respect to $r$ in the second inequality and (\ref{iwRp}) in the third inequality. In particular, for any $t \in\left[\frac{\pi}{2 \sqrt{K_{r}}}, \frac{5 \pi}{6 \sqrt{K_{r}}}\right]$, the inequality (\ref{A(r)}) still holds and $\sqrt{K_{r}} t\in \left[\frac{\pi}{2}, \frac{5 \pi}{6}\right]$. Since the sine function is strictly decreasing on $\left[\frac{\pi}{2}, \frac{5 \pi}{6}\right]$, we have $\sin \left(\sqrt{K_{r}} t\right) \geq \sin \frac{5\pi}{6}=\frac{1}{2}$.
Then we obtain
\beqn
\left(A(r)\right)^{\frac{1}{2p-1}} &\leq & 2^{\frac{n-1}{2p-1}} A(t)^{\frac{1}{2p-1}} \left(e^{\vartheta r}\right)^{\frac{1}{2 p-1}} \epsilon^{\frac{n-1}{2p-1}} \nonumber\\
&& +\left(\frac{n-1}{(2p-1)(2p-n)}\right)^{\frac{p-1}{2p-1}}  \left(e^{\vartheta r}\right)^{\frac{1}{2p-1}} m\left(B_{R_{0}}^{+}(x_{0})\right)^{\frac{1}{2p-1}} \frac{\pi r}{4} \epsilon^{\frac{n-1}{2 p-1}}.
\eeqn
Now, by using $(a+b)^{2p-1} \leq 2^{2p-2}(a^{2p-1}+b^{2p-1})$ for $a, b>0$ and $p>1$, we obtain the following
\beq\label{A(r)2}
A(r) &\leq & 2^{2 p-2}\Big[2^{n-1} A(t) e^{\vartheta r} \epsilon^{n-1} \nonumber\\
&& +\left(\frac{n-1}{(2 p-1)(2 p-n)}\right)^{p-1} e^{\vartheta r} m\left(B_{R_{0}}^{+}(x_{0})\right) \left(\frac{\pi r}{4}\right)^{2 p-1} \epsilon^{n-1}\Big]\nonumber\\
&\leq & 2^{2 p+n-3} A(t) e^{\vartheta r} \epsilon^{n-1}\nonumber\\
&& +\left(\frac{n-1}{(2 p-1)(2 p-n)}\right)^{p-1} e^{\vartheta r}\frac{\pi^{2 p-1}}{4^{p}} m\left(B_{R_{0}}^{+}(x_{0})\right)r^{2p-1} \epsilon^{n-1}.
\eeq
Furthermore, by applying the First Mean-value Theorem for definite integral to $A(t)$ on $\left[\frac{\pi}{2 \sqrt{K_{r}}}, \frac{5 \pi}{6 \sqrt{K_{r}}}\right]$, we can deduce that there exists a $t_{0} \in\left[\frac{\pi}{2 \sqrt{K_{r}}}, \frac{5 \pi}{6 \sqrt{K_{r}}}\right]$ such that
\be
A(t_{0})=\frac{3}{\pi} \sqrt{K_{r}} \int_{\frac{\pi}{2 \sqrt{K_{r}}}}^{\frac{5 \pi}{6 \sqrt{K_{r}}}} A(t) d t \leq \frac{3}{r} \int_{0}^{R_{0}} A(t) d t = \frac{3}{r} m\left(B_{R_{0}}^{+}(x_{0})\right). \label{intmean}
\ee
Since (\ref{A(r)2}) is valid for all $t \in\left[\frac{\pi}{2 \sqrt{K_{r}}}, \frac{5 \pi}{6 \sqrt{K_{r}}}\right]$, it is particularly valid for $t_{0}$ in (\ref{intmean}). Hence, we have
\beqn
A(r) &\leq& 2^{2 p+n-3} \left(\frac{3}{r} m\left(B_{R_{0}}^{+}(x_{0})\right)\right) e^{\vartheta r} \epsilon^{n-1} \\
&& +\left(\frac{n-1}{(2 p-1)(2 p-n)}\right)^{p-1} e^{\vartheta r}\frac{\pi^{2 p-1}}{4^{p}} m\left(B_{R_{0}}^{+}(x_{0})\right)r^{2p-1} \epsilon^{n-1}\\
&\leq& 2^{2 p+n-3} \left(\frac{3\Lambda_{F}}{\pi} m\left(B_{R_{0}}^{+}(x_{0})\right)\right) e^{\vartheta r} \epsilon^{n-1}\\
&& +\left(\frac{n-1}{(2 p-1)(2 p-n)}\right)^{p-1} e^{\vartheta r}\frac{\pi^{2 p-1}}{4^{p}} m\left(B_{R_{0}}^{+}(x_{0})\right)\Lambda_{F}^{-(2p-1)}R_{0}^{2p-1} \epsilon^{n-1}\\
&=:& C_{1}(n,p,\Lambda_{F},R_{0}) e^{\vartheta r} m\left(B_{R_{0}}^{+}(x_{0})\right) \epsilon^{n-1},
\eeqn
where $C_{1}(n,p,\Lambda_{F},R_{0}):=\frac{3\cdot 2^{2p+n-3}\Lambda_{F}}{\pi} +\frac{\pi^{2p-1}}{4^{p}}\left(\frac{n-1}{(2p-1)(2 p-n)}\right)^{p-1} \Lambda_{F}^{-(2p-1)}R_{0}^{2p-1}$. By $\epsilon \leq \frac{\pi}{6}$, we obtain (\ref{lem4.2}).  This completes the proof of Lemma \ref{S(r)}.
\end{proof}

\vskip 2mm

In the following, let
\be
v(n,K,r,\vartheta):={\rm {Vol}}(\mathbb{S}^{n-1})\int_{0}^{r} e^{\vartheta t}s_{K}^{n-1}(t)dt, \label{vnkrv}
\ee
where ${\rm {Vol}}(\mathbb{S}^{n-1})$ denotes the Euclidean volume of the unit sphere $\mathbb{S}^{n-1}$ in $\mathbb{R}^n$. Following the argument of Theorem 3.2 in \cite{ChF}, we can get the following relative volume comparison easily from Theorem \ref{Lap1}.

\begin{thm}\label{vol}
Let $(M, F, m)$ be an $n$-dimensional Finsler metric measure manifold. Assume that $\mathbf{S}\geq -\vartheta$ for some $\vartheta \geq 0$.
Then for $p>\frac{n}{2}$, $K \in \mathbb{R}$ and $0<r \leq R$ $(R \leq \frac{\pi}{2 \sqrt{K}}$ when $K>0)$, there exists a constant $C_{2}=C_{2}(n, p, \vartheta, K, R)>0$ such that
\be
\left(\frac{m\left(B_{R}^{+}(x_{0})\right)}{v(n, K, R, \vartheta)}\right)^{\frac{1}{2 p-1}}-\left(\frac{m\left(B_{r}^{+}(x_{0})\right)}{v(n, K, r, \vartheta)}\right)^{\frac{1}{2 p-1}} \leq C_{2} \left\|{\rm{Ric}}_{\infty}^{K}\right\|_{p, R, \vartheta}^{\frac{p}{2 p-1}}(x_{0}), \label{volmv}
\ee
where $C_{2}(n, p, \vartheta, K, R):=\left(\frac{n-1}{(2 p-1)(2 p-n)}\right)^{\frac{p-1}{2 p-1}} {\rm{Vol}}\left(\mathbb{S}^{n-1}\right) \int_{0}^{R} e^{\frac{2 p}{2 p-1}\vartheta t} s_{K}^{n-1}(t)\left(\frac{t}{v(n, K, t, \vartheta)}\right)^{\frac{2 p}{2 p-1}} d t$ is nondecreasing in $R$.
\end{thm}

\vskip 2mm

When $K=0$, $\vartheta=0$, write
$$
v(n, r):=v(n, 0, r, 0)={\rm {Vol}}(\mathbb{S}^{n-1}) \int_{0}^{r} t^{n-1} d t={\rm {Vol}}(\mathbb{S}^{n-1})\frac{r^{n}}{n}
$$
and
\beqn
C_{2}(n, p, r)&:=&C_{2}(n, p, 0, 0, r)\\
&=&\left(\frac{n-1}{(2 p-1)(2 p-n)}\right)^{\frac{p-1}{2p-1}} {\rm {Vol}}(\mathbb{S}^{n-1}) \int_{0}^{r} t^{n-1}\left(\frac{t}{{\rm {Vol}}(\mathbb{S}^{n-1})\frac{t^{n}}{n}}\right)^{\frac{2p}{2p-1}} d t\\
&=&\left(\frac{n^{2p}(n-1)^{p-1}(2p-1)^{p}}{(2p-n)^{3p-2}}\right)^{\frac{1}{2p-1}}{\rm {Vol}}(\mathbb{S}^{n-1})^{-\frac{1}{2p-1}} r^{\frac{2p-n}{2p-1}}.
\eeqn
By a simple calculation, we have
$$
C_{2}(n, p, r)^{2p-1}v(n, r)=\frac{n^{2p-1}(n-1)^{p-1}(2p-1)^{p}}{(2p-n)^{3p-2}}\cdot r^{2p}.
$$
Then we can obtain the following Bishop-Gromov type volume comparison.

\begin{thm}\label{BG}
Let $(M, F, m)$ be an $n$-dimensional Finsler metric measure manifold. Assume that $\mathbf{S} \geq 0$. Given $p>\frac{n}{2}$, $\Xi>1$ and $0<r_{1} \leq r_{2} \leq R$.  Then there exists a constant $\varsigma >0$, such that when $\overline{\left\|{\rm{Ric}}_{\infty}^{0}\right\|}_{p, R, 0}\leq \varsigma$,
the following inequality holds
\be\label{Bishop}
\frac{m(B^{+}_{r_{2}}(x_{0}))}{m(B^{+}_{r_{1}}(x_{0}))} \leq \Xi\left(\frac{r_{2}}{r_{1}}\right)^{n},
\ee
where $\varsigma:=\Big(\frac{1-\Xi^{-\frac{1}{2 p-1}}}{2 C_{2}(n,p,R)}\Big)^{\frac{2p-1}{p}} v(n,R)^{-\frac{1}{p}}$.
\end{thm}
\begin{proof}
For $0<r_{1} \leq r_{2} \leq R$, it follows from Theorem \ref{vol} that
$$
\left(\frac{v(n, r_{1})}{v(n, r_{2})}\right)^{\frac{1}{2 p-1}}-
\left(\frac{m\left(B_{r_{1}}^{+}(x_{0})\right)}{m\left(B_{r_{2}}^{+}(x_{0})\right)}\right)^{\frac{1}{2 p-1}} \leq C_{2}(n, p, r_{2})\left\|{\rm{Ric}}_{\infty}^{0}\right\|_{p, r_{2}, 0}^{\frac{p}{2 p-1}}(x_{0})\left(\frac{v(n, r_{1})}{m\left(B_{r_{2}}^{+}(x_{0})\right)}\right)^{\frac{1}{2 p-1}}.
$$
The inequality can be rewritten as
$$
\frac{m(B^{+}_{r_{1}}(x_{0}))}{m(B^{+}_{r_{2}}(x_{0}))} \geq (1-c)^{2 p-1} \frac{v(n, r_{1})}{v(n, r_{2})},
$$
where $c:=\left(\frac{v(n, r_{2})}{m(B^{+}_{r_{2}}(x_{0}))}\right)^{\frac{1}{2p-1}} C_{2}(n,p,r_{2})\left\|{\rm{Ric}}_{\infty}^{0}\right\|_{p, r_{2}, 0}^{\frac{p}{2p-1}}(x_{0})$.
In order to estimate $c$, we use Theorem \ref{vol} again and obtain
\beqn
\left(\frac{v(n, r_{2})}{m(B^{+}_{r_{2}}(x_{0}))}\right)^{\frac{1}{2 p-1}}
& \leq & \left[\left(\frac{m(B_{R}^{+}(x_{0}))}{v(n, R)}\right)^{\frac{1}{2p-1}} - C_{2}(n,p,R) \left\|{\rm{Ric}}_{\infty}^{0}\right\|_{p, R, 0}^{\frac{p}{2p-1}}(x_0)\right]^{-1} \\
& = & \left(\frac{v(n, R)}{m(B_{R}^{+}(x_0))}\right)^{\frac{1}{2p-1}}\left[1-v(n, R)^{\frac{1}{2p-1}} C_{2}(n,p,R) \overline{\left\|{\rm{Ric}}_{\infty}^{0}\right\|}_{p, R, 0}^{\frac{p}{2p-1}}(x_{0})\right]^{-1}.
\eeqn
Now, choose $\varsigma_{1}(n,p,R):=\left( 2 C_{2}(n,p,R)\right)^{-\frac{2 p-1}{p}} v(n, R)^{-\frac{1}{p}}$. Then, when $\overline{\left\|{\rm{Ric}}_{\infty}^{0}\right\|}_{p, R, 0}\leq\varsigma_{1}$, we have
$$
\left(\frac{v(n, r_{2})}{m(B^{+}_{r_{2}}(x_0))}\right)^{\frac{1}{2 p-1}} \leq 2\left(\frac{v(n, R)}{m(B^{+}_{R}(x_{0}))}\right)^{\frac{1}{2 p-1}}.
$$
In this case, by the definition of $c$, we can get
\beqn
c & \leq & 2\left(\frac{v(n, R)}{m(B^{+}_{R}(x_{0}))}\right)^{\frac{1}{2 p-1}} C_{2}(n,p,r_{2})\left\|{\rm{Ric}}_{\infty}^{0}\right\|_{p, r_{2}, 0}^{\frac{p}{2p-1}}(x_{0})\\
  & \leq & 2 v(n, R)^{\frac{1}{2p-1}} C_{2}(n,p,R) \overline{\left\|{\rm{Ric}}_{\infty}^{0}\right\|}_{p, R, 0}^{\frac{p}{2p-1}}(x_{0}).
\eeqn
Furthermore, for $\Xi>1$, choose
$\varsigma_{2}:=\Big(\frac{1-\Xi^{-\frac{1}{2 p-1}}}{2 C_{2}(n,p,R)}\Big)^{\frac{2p-1}{p}} v(n,R)^{-\frac{1}{p}}$. Then, when $\overline{\left\|{\rm{Ric}}_{\infty}^{0}\right\|}_{p, R, 0} \leq \varsigma_{2}$,
we have $c\leq 1-\Xi^{-\frac{1}{2 p-1}}$.

Finally, choose $\varsigma:=\min\{\varsigma_{1}, \varsigma_{2}\}=\varsigma_{2}$. Then, when $\overline{\left\|{\rm{Ric}}_{\infty}^{0}\right\|}_{p, R, 0} \leq \varsigma$, we can obtain the following
$$
\frac{m(B^{+}_{r_{2}}(x_{0}))}{m(B^{+}_{r_{1}}(x_{0}))} \leq \Xi\frac{v(n, r_{2})}{v(n, r_{1})} =\Xi\left(\frac{r_{2}}{r_{1}}\right)^{n}.
$$
This completes the proof.
\end{proof}

\vskip 2mm

Next, we will establish a volume comparison theorem controlled by the integral weighted Ricci curvature  for nonconcentric balls, which will play a crucial role in the proof of Theorem \ref{myers}.

\begin{thm}\label{noncon}
Let $(M, F, m)$ be an $n$-dimensional $(n \geq 2)$ forward complete Finsler metric measure manifold with $\Lambda_F<\infty$. Assume that $\mathbf{S}\geq -\vartheta$ for some $\vartheta \geq 0$.
Given constants $p \in (\frac{n}{2}, n]$, $\alpha\in [\frac{1}{2},1)$ and a fixed point $x_{0}\in M$. Then,  for $0 \leq r < R_{0}$ and an arbitrarily fixed point $x\in M$ satisfying that $d_{F}(x_{0}, x) \leq \Lambda_{F}^{-1}\left( R_{0}-r\right)$, there exist positive constants $D_{1}=D_{1}(n, p, \vartheta , R_{0})$ and $D_{2}=D_{2}(n,p,\Lambda_{F},\vartheta,\alpha,R_{0})$, such that, when $\overline{\left\|{\rm{Ric}}_{\infty}^{0}\right\|}_{p, R_{0}, \vartheta}<D_{1}^{-\frac{2p-1}{p}}$, we have
\beq
&& \left(\frac{m\left(B_{\Lambda_{F}^{-1}r}^{+}(x)\right)}{m\left(B_{R_{0}}^{+}(x_{0})\right)}\right)^{\frac{1}{2 p-1}} \geq  \Big[\frac{1}{4} \left(1-D_{1} \overline{\left\|{\rm{Ric}}_{\infty}^{0}\right\|}_{p, R_{0}, \vartheta}^{\frac{p}{2p-1}}\right)
e^{-\frac{\vartheta R_{0}}{2 p-1}}\Lambda_{F}^{-\left(2 \vartheta R_{0}b(p,\alpha)+\frac{5n}{2p-1}\right)}   \nonumber \\
&& \times \left(\frac{r}{R_{0}}\right)^{\vartheta R_{0} b(p,\alpha)+\frac{4n-2p}{2p-1}}  -  \frac{D_{2}\overline{\left\|{\rm{Ric}}_{\infty}^{0}\right\|}_{p, R_{0}, \vartheta}^{\frac{p}{2p-1}}}{1-D_{1} \overline{\left\|{\rm{Ric}}_{\infty}^{0}\right\|}_{p, R_{0}, \vartheta}^{\frac{p}{2p-1}}} \Big] \left(\frac{r}{R_{0}}\right)^{\frac{2p-n}{2p-1}}, \label{nconcom}
\eeq
where
\beqn
&& D_{1}(n,p,\vartheta,R_{0}):=D(n, p) e^{\frac{2 p+1}{2 p-1} \vartheta R_{0}} R_{0}^{\frac{2p}{2 p-1}},  \ \ \ \ \ \ \ \ \ D(n, p):=\frac{n(n-1)^{\frac{p-1}{2 p-1}}(2 p-1)^{\frac{p}{2 p-1}}}{(2 p-n)^\frac{3p-2}{2p-1}},    \\
&& D_{2}(n,p,\Lambda_{F},\vartheta,\alpha,R_{0}):=\frac{D_{1}(n,p,\vartheta,R_{0}) e^{\frac{\vartheta R_{0}}{2 p-1}} \Lambda_{F}^{-\frac{2p-n}{2p-1}}}{1-e^{-\frac{\vartheta R_{0}}{2p-1}}(2-\alpha)^{\frac{2p-2n}{2p-1}}\alpha^{\frac{n}{2p-1}} }, \ \ \  b(p,\alpha):=\frac{1}{(2p-1)\ln(2-\alpha)}.
\eeqn
\end{thm}
\begin{proof}
Recalling the definitions of $v(n,K,R,\vartheta)$ and $C_{2}(n,p,\vartheta,K,R)$ with $K=0$, we can obtain the following for $0< r \leq R$
\be
\frac{v(n,0,R,\vartheta)}{v(n,0,r,\vartheta)}=\frac{{\rm {Vol}}(\mathbb{S}^{n-1}) \int_{0}^{R} e^{\vartheta t}t^{n-1} d t}{{\rm {Vol}}(\mathbb{S}^{n-1}) \int_{0}^{r} e^{\vartheta t} t^{n-1} d t} \leq e^{\vartheta R} \left(\frac{R}{r}\right)^{n} \label{v0}
\ee
and
\beq
C_{2}(n,p,\vartheta,0,R)v(n,0,R,\vartheta)^{\frac{1}{2p-1}} &\leq & \frac{n(n-1)^{\frac{p-1}{2p-1}}(2 p-1)^{\frac{p}{2p-1}}}{(2 p-n)^{\frac{3 p-2}{2p-1}}} e^{\frac{2 p+1}{2 p-1} \vartheta R } R ^{\frac{2p}{2 p-1}}\nonumber\\
&=& D(n,p) e^{\frac{2 p+1}{2 p-1} \vartheta R } R ^{\frac{2p}{2 p-1}}:= D_{1}(n,p,\vartheta, R). \label{C2v}
\eeq
In particular, by the assumption, for $0<R \leq R_{0}$, we have
\be
D_{1}(n,p,\vartheta,R_{0})< \overline{\left\|{\rm{Ric}}_{\infty}^{0}\right\|}_{p, R_{0}, \vartheta}^{- \frac{p}{2p-1}}, \ \ \ D_{1}(n,p,\vartheta,R)\leq D_{1}(n,p,\vartheta,R_{0})R_{0}^{-\frac{2p}{2p-1}} R_{}^{\frac{2p}{2p-1}}. \label{C2VRic}
\ee

In the following,  we will divide the proof into four steps.

$\mathbf{Step\ 1.}$ \ For any $z \in B_{R_{0}}^{+}(x_{0})$ and $0<r \leq R \leq R_{0}-\Lambda_{F}d_{F}(x_{0}, z)$, it is easy to see that $B_{R}^{+}(z) \subset B_{R_{0}}^{+}(x_{0})$. Applying (\ref{volmv}) with $K=0$ to the geodesic balls centered at $z$ yields the following inequality
\beq
m\left(B_{R}^{+}(z)\right)^{\frac{1}{2 p-1}}
&\leq & C_{2}(n, p, \vartheta, 0, R) v(n, 0, R, \vartheta)^{\frac{1}{2 p-1}} \left\|{\rm{Ric}}_{\infty}^{0}\right\|_{p, R, \vartheta}^{\frac{p}{2 p-1}}(z)\nonumber\\
&& + \left(\frac{v(n, 0, R, \vartheta)}{v(n, 0, r, \vartheta)}\right)^{\frac{1}{2 p-1}} m\left(B_{r}^{+}(z)\right)^{\frac{1}{2 p-1}}. \label{Bz}
\eeq
Further, dividing both sides of (\ref{Bz}) by $\left(m\left(B_{R_{0}}^{+}(x_{0})\right)\right)^{\frac{1}{2p-1}}$, we obtain
\beq
\left(\frac{m(B_{R}^{+}(z))}{m(B_{R_0}^{+}(x_{0}))}\right)^{\frac{1}{2p-1}}
&\leq& C_{2}(n,p,\vartheta,0,R) \left(\frac{v(n,0,R,\vartheta)}{m(B^{+}_{R_{0}}(x_{0}))}\right)^{\frac{1}{2p-1}} \left\|{\rm{Ric}}_{\infty}^{0}\right\|^{\frac{p}{2p-1}}_{p, R, \vartheta}(z) \nonumber\\
&&+\left(\frac{m(B_{r}^{+}(z))}{m(B_{R_0}^{+}(x_{0}))}\right)^{\frac{1}{2p-1}} \left(\frac{v(n,0,R,\vartheta)}{v(n,0,r,\vartheta)}\right)^{\frac{1}{2p-1}} \nonumber\\
&\leq& C_{2}(n,p,\vartheta,0,R) \left(\frac{v(n,0,R,\vartheta)}{m(B^{+}_{R}(z))}\right)^{\frac{1}{2p-1}} \left\|{\rm{Ric}}_{\infty}^{0}\right\|^{\frac{p}{2p-1}}_{p, R, \vartheta}(z) \nonumber\\
&&+\left(\frac{m(B_{r}^{+}(z))}{m(B_{R_0}^{+}(x_{0}))}\right)^{\frac{1}{2p-1}} \left(\frac{v(n,0,R,\vartheta)}{v(n,0,r,\vartheta)}\right)^{\frac{1}{2p-1}} \nonumber\\
&\leq& C_{2}(n,p,\vartheta,0,R)v(n,0,R,\vartheta)^{\frac{1}{2p-1}} \overline{\left\|{\rm{Ric}}_{\infty}^{0}\right\|}_{p, R, \vartheta}^{\frac{p}{2p-1}} \nonumber\\
&&+\left(\frac{m(B_{r}^{+}(z))}{m(B_{R_0}^{+}(x_{0}))}\right)^{\frac{1}{2p-1}} \left(\frac{v(n,0,R,\vartheta)}{v(n,0,r,\vartheta)}\right)^{\frac{1}{2p-1}} \nonumber\\
&\leq & C_{2}(n,p,\vartheta,0,R)v(n,0,R,\vartheta)^{\frac{1}{2p-1}} \sup_{x\in M} \left(\frac{m(B_{R_{0}}^{+}(x))}{m(B_{R}^{+}(x))}\right)^{\frac{1}{2p-1}} \overline{\left\|{\rm{Ric}}_{\infty}^{0}\right\|}_{p, R_{0}, \vartheta}^{\frac{p}{2p-1}}  \nonumber\\ &&+\left(\frac{m(B_{r}^{+}(z))}{m(B_{R_0}^{+}(x_{0}))}\right)^{\frac{1}{2p-1}} \left(\frac{v(n,0,R,\vartheta)}{v(n,0,r,\vartheta)}\right)^{\frac{1}{2p-1}},\label{BRzx0}
\eeq
where we have used (\ref{iwRp}) and (\ref{iwRM}) in the third inequality and (\ref{ric<}) in the fourth inequality.
Moreover, by (\ref{volmv}) again, we have
\beqn
\left(\frac{m(B_{R_{0}}^{+}(x))}{m(B_{R}^{+}(x))}\right)^{\frac{1}{2p-1}}
&\leq& C_{2}(n,p,\vartheta,0,R_{0}) \left(\frac{v(n,0,R_{0},\vartheta)}{m(B^{+}_{R}(x))}\right)^{\frac{1}{2p-1}} \left\|{\rm{Ric}}_{\infty}^{0}\right\|^{\frac{p}{2p-1}}_{p, R_{0}, \vartheta}(x) \nonumber\\
&& +\left(\frac{v(n,0,R_{0},\vartheta)}{v(n,0,R,\vartheta)}\right)^{\frac{1}{2p-1}}\nonumber\\
&=& C_{2}(n,p,\vartheta,0,R_{0}) v(n,0,R_{0},\vartheta)^{\frac{1}{2p-1}} \left(\frac{m(B_{R_{0}}^{+}(x))}{m(B_{R}^{+}(x))}\right)^{\frac{1}{2p-1}}\nonumber\\
&& \times \overline{\left\|{\rm{Ric}}_{\infty}^{0}\right\|}_{p, R_{0}, \vartheta}^{\frac{p}{2p-1}}
+ \left(\frac{v(n,0,R_{0},\vartheta)}{v(n,0,R,\vartheta)}\right)^{\frac{1}{2p-1}},
\eeqn
from which, we can obtain
\be
\left(\frac{m(B_{R_{0}}^{+}(x))}{m(B_{R}^{+}(x))}\right)^{\frac{1}{2p-1}}
\leq \left(\frac{v(n,0,R_{0},\vartheta)}{v(n,0,R,\vartheta)}\right)^{\frac{1}{2p-1}} \left(1-C_{2}(n,p,\vartheta,0,R_{0}) v(n,0,R_{0},\vartheta)^{\frac{1}{2p-1}} \overline{\left\|{\rm{Ric}}_{\infty}^{0}\right\|}_{p, R_{0}, \vartheta}^{\frac{p}{2p-1}}\right)^{-1}. \label{volBR}
\ee
Here, from (\ref{C2v}) and (\ref{C2VRic}), we can assert that
\[
1-C_{2}(n,p,\vartheta,0,R_{0}) v(n,0,R_{0},\vartheta)^{\frac{1}{2p-1}} \overline{\left\|{\rm{Ric}}_{\infty}^{0}\right\|}_{p, R_{0}, \vartheta}^{\frac{p}{2p-1}}>0.
\]
Hence, from (\ref{BRzx0}) and (\ref{volBR}), we get
\beq
\left(\frac{m(B_{R}^{+}(z))}{m(B_{R_0}^{+}(x_{0}))}\right)^{\frac{1}{2p-1}}
&\leq& C_{2}(n,p,\vartheta,0,R)v(n,0,R,\vartheta)^{\frac{1}{2p-1}} \left(\frac{v(n,0,R_{0},\vartheta)}{v(n,0,R,\vartheta)}\right)^{\frac{1}{2p-1}}\nonumber \\
&&\times\left(1-C_{2}(n,p,\vartheta,0,R_{0}) v(n,0,R_{0},\vartheta)^{\frac{1}{2p-1}} \overline{\left\|{\rm{Ric}}_{\infty}^{0}\right\|}_{p, R_{0}, \vartheta}^{\frac{p}{2p-1}}\right)^{-1} \overline{\left\|{\rm{Ric}}_{\infty}^{0}\right\|}_{p, R_{0}, \vartheta}^{\frac{p}{2p-1}}\nonumber\\ &&+\left(\frac{m(B_{r}^{+}(z))}{m(B_{R_0}^{+}(x_{0}))}\right)^{\frac{1}{2p-1}} \left(\frac{v(n,0,R,\vartheta)}{v(n,0,r,\vartheta)}\right)^{\frac{1}{2p-1}}. \label{volz}
\eeq

Then, from (\ref{volz}) and by (\ref{v0}), (\ref{C2v}) and (\ref{C2VRic}), we have
\beq
\left(\frac{m(B_{R}^{+}(z))}{m(B_{R_0}^{+}(x_{0}))}\right)^{\frac{1}{2p-1}}
&\leq& \frac{D_{1}(n, p,\vartheta,R_{0}) e^{\frac{\vartheta R_{0}}{2p-1}} R_{0}^{\frac{n-2p}{2 p-1}} R^{\frac{2p-n}{2p-1}} \overline{\left\|{\rm{Ric}}_{\infty}^{0}\right\|}_{p, R_0, \vartheta}^{\frac{p}{2p-1}}}{1-D_{1}(n, p,\vartheta,R_{0}) \overline{\left\|{\rm{Ric}}_{\infty}^{0}\right\|}_{p, R_0, \vartheta}^{\frac{p}{2p-1}}} \nonumber \\
&& +\left(\frac{m\left(B_{r}^{+}(z)\right)}{m\left(B_{R_{0}}^{+}(x_{0})\right)}\right)^{\frac{1}{2p-1}}  e^{\frac{\vartheta R_{0}}{2p-1}}\left(\frac{R}{r}\right)^{\frac{n}{2p-1}}. \label{volzz}
\eeq

$\mathbf{Step\ 2.}$ \ Let $\gamma:[0, d_{F}(x_{0},x)] \rightarrow M$ be the minimizing geodesic from $x_{0}=\gamma(0)$ to $x=\gamma(d_{F}(x_{0},x))$.
Take $k=\bigg\lfloor\frac{\ln \Big(1+\tfrac{d_{F}(x_{0}, x)}{\Lambda_{F}^{-2} r}\Big)}{\ln (2-\alpha)}\bigg\rfloor+2$, where $\lfloor s \rfloor$ denotes the floor function, which takes the greatest integer less than or equal to the real number $s$.
Let $t_{i}:=d_{F}(x_{0}, x)-\Lambda_{F}^{-2}\left((2-\alpha)^{i-1}-1\right)r$,  $1\leq i\leq k-1$. We will prove that $t_{i} \in [0, d_{F}(x_{0},x)]$ for  $1\leq i\leq k-1$.

Because $2-\alpha\in (1,\frac{3}{2}]$, it is obvious that $t_{i}\leq d_{F}(x_{0},x)$. Moreover, because the sequence $\{t_{i}\}_{i=1}^{k-1}$ is strictly decreasing in $i$, so it suffices to verify $t_{k-1}\geq 0$.
Actually, by the definition of $k$ and the fact that $\lfloor s \rfloor \leq s\leq \lfloor s \rfloor +1$, we have
\be
k-2 = \bigg\lfloor\frac{\ln \Big(1+\tfrac{d_{F}(x_{0}, x)}{\Lambda_{F}^{-2} r}\Big)}{\ln (2-\alpha)}\bigg\rfloor
\leq \frac{\ln \Big(1+\frac{d_{F}(x_{0}, x)}{\Lambda_{F}^{-2} r}\Big)}{\ln (2-\alpha)}
\leq \bigg\lfloor\frac{\ln \Big(1+\tfrac{d_{F}(x_{0}, x)}{\Lambda_{F}^{-2} r}\Big)}{\ln (2-\alpha)}\bigg\rfloor +1 = k-1. \label{floor}
\ee
By using $\log_{b} a = \frac{\ln a}{\ln b}$ for $a, b >0$ and $b\neq 1$, we can rewrite above inequality as
$$
k-2 \leq \log_{(2-\alpha)}\Big(1+\frac{d_{F}(x_{0}, x)}{\Lambda_{F}^{-2} r}\Big) \leq k-1.
$$
Since $2-\alpha\in (1,\frac{3}{2}]$, the function $(2-\alpha)^t$ is increasing in $t$. Therefore, we obtain
$$
(2-\alpha)^{k-2}\leq (2-\alpha)^{\log_{(2-\alpha)}\Big(1+\frac{d_{F}(x_{0}, x)}{\Lambda_{F}^{-2} r}\Big)} \leq (2-\alpha)^{k-1},
$$
which means that
\be
(2-\alpha)^{k-2}\leq 1+\frac{d_{F}(x_{0}, x)}{\Lambda_{F}^{-2} r}\leq (2-\alpha)^{k-1}. \label{k}
\ee
From the first inequality of (\ref{k}), we have $t_{k-1} = d_{F}(x_{0}, x)-\Lambda_{F}^{-2}((2-\alpha)^{k-2}-1)r \geq 0$. Thus we can conclude that $0\leq t_{i}\leq d_{F}(x_{0},x)$  for $1\leq i\leq k-1$.
Furthermore, it is reasonable to let $x_{i}:=\gamma(t_{i})$ for $1 \leq i \leq k-1$ and  $x_{k}:=x_{0}$.

\vskip 2mm
$\mathbf{Step\ 3.}$ \ For $1\leq i\leq k$, define $r_{i}:=\Lambda_{F}^{-1}\alpha(2-\alpha)^{i-2}r$, $R_{i}:=\Lambda_{F}^{-1}(2-\alpha)^{i-1}r$. We claim that $B_{r_{i+1}}^{+}(x_{i+1}) \subset B_{R_{i}}^{+}(x_{i})$ for $1 \leq i \leq k-1$.

Firstly, for $1\leq i\leq k-2$ and any $\bar{x}\in B_{r_{i+1}}^{+}(x_{i+1})$, we have
\beqn
d_{F}(x_{i},\bar{x})&\leq& d_{F}(x_{i},x_{i+1})+d_{F}(x_{i+1},\bar{x})\leq \Lambda_{F}d_{F}(x_{i+1},x_{i})+ r_{i+1}\\
&\leq & \Lambda_{F}^{-1}\left[((2-\alpha)^{i}-1)r-((2-\alpha)^{i-1}-1)r\right]+r_{i+1}\\
&= & \Lambda_{F}^{-1}(1-\alpha)(2-\alpha)^{i-1}r+\Lambda_{F}^{-1}\alpha(2-\alpha)^{i-1}r = \Lambda_{F}^{-1}(2-\alpha)^{i-1}r = R_{i}.
\eeqn
Thus $\bar{x}\in B_{R_{i}}^{+}(x_{i})$, which means that $B_{r_{i+1}}^{+}(x_{i+1}) \subset B_{R_{i}}^{+}(x_{i})$ for $1\leq i\leq k-2$.

Next, we consider the case when $i=k-1$. Recalling that $x_{k} =x_{0}$ and noticing that the second inequality of (\ref{k}) implies $d_{F}(x_{0},x)\leq\Lambda_{F}^{-2}\left((2-\alpha)^{k-1}-1\right)r$, for any $\bar{x}\in B_{r_{k}}^{+}(x_{0})$, we can get
\beqn
d_{F}(x_{k-1},\bar{x})&\leq& d_{F}(x_{k-1},x_{0})+d_{F}(x_{0},\bar{x})\leq \Lambda_{F}d_{F}(x_{0},x_{k-1})+d_{F}(x_{0},\bar{x})\\
&\leq & \Lambda_{F}\left[d_{F}(x_{0}, x)-\Lambda_{F}^{-2}\left((2-\alpha)^{k-2}-1\right)r\right]+r_{k}\\
&\leq & \Lambda_{F}\left[\Lambda_{F}^{-2}\left((2-\alpha)^{k-1}-1\right)r-\Lambda_{F}^{-2}\left((2-\alpha)^{k-2}-1\right)r\right]+r_{k}\\
&=& \Lambda_{F}^{-1}(1-\alpha)(2-\alpha)^{k-2}r+\Lambda_{F}^{-1}\alpha(2-\alpha)^{k-2}r = \Lambda_{F}^{-1}(2-\alpha)^{k-2}r = R_{k-1},
\eeqn
which means $B_{r_{k}}^{+}(x_{0}) \subset B_{R_{k-1}}^{+}(x_{k-1})$.

In sum, we conclude that $B_{r_{i+1}}^{+}\left(x_{i+1}\right) \subset B_{R_{i}}^{+}\left(x_{i}\right)$ for $1 \leq i \leq k-1$.

\vskip 2mm

$\mathbf{Step\ 4.}$ \ In the following, we will apply the iteration trick to get the volume comparison for nonconcentric balls.

By the discussions in {\bf Step 2} and {\bf Step 3}, we know that
\beqn
d_{F}(x_{0}, x_{i})&=& t_{i} =d_{F}(x_{0}, x)-\Lambda_{F}^{-2}\left((2-\alpha)^{i-1}-1\right)r = d_{F}(x_{0}, x)-\Lambda_{F}^{-2} (2-\alpha)^{i-1} r + \Lambda_{F}^{-2}r \\
&\leq & \Lambda_{F}^{-1}(R_{0}-r)- \Lambda_{F}^{-1}R_{i}+\Lambda_{F}^{-1}r = \Lambda_{F}^{-1}(R_{0}-R_{i}).
\eeqn
Hence, $x_{i}$ and $r_{i}, \ R_{i}$ satisfy that $0<r_{i}\leq R_{i}\leq R_{0}-\Lambda_{F}d_{F}(x_{0},x_{i})$  for $1\leq i\leq k-1$. By the discussion in {\bf Step 1},
we have $B_{R_{i}}^{+}\left(x_{i}\right)\subset B_{R_{0}}^{+}(x_{0})$ for $1\leq i\leq k-1$. Further, for $2\leq i\leq k-1$, we can rewrite (\ref{volzz}) as
\beq
\left(\frac{m\left(B_{R_{i}}^{+}\left(x_{i}\right)\right)}{m\left(B_{R_{0}}^{+}(x_{0})\right)}\right)^{\frac{1}{2 p-1}}
&\leq& \frac{D_{1}(n, p,\vartheta,R_{0}) e^{\frac{\vartheta R_{0}}{2p-1}} R_{0}^{\frac{n-2p}{2 p-1}} R_{i}^{\frac{2p-n}{2p-1}}\overline{\left\|{\rm{Ric}}_{\infty}^{0}\right\|}_{p, R_{0}, \vartheta}^{\frac{p}{2p-1}}}{1-D_{1}(n, p,\vartheta,R_{0}) \overline{\left\|{\rm{Ric}}_{\infty}^{0}\right\|}_{p, R_{0}, \vartheta}^{\frac{p}{2p-1}}}\nonumber\\
&& +\left(\frac{m\left(B_{r_{i}}^{+}(x_{i})\right)}{m\left(B_{R_{0}}^{+}(x_{0})\right)}\right)^{\frac{1}{2p-1}}  e^{\frac{\vartheta R_{0}}{2p-1}}\left(\frac{R_{i}}{r_{i}}\right)^{\frac{n}{2p-1}}\nonumber\\
&\leq & \frac{D_{1}(n, p,\vartheta,R_{0}) e^{\frac{\vartheta R_{0}}{2p-1}} R_{0}^{\frac{n-2p}{2 p-1}} R_{i}^{\frac{2p-n}{2p-1}}\overline{\left\|{\rm{Ric}}_{\infty}^{0}\right\|}_{p, R_{0}, \vartheta}^{\frac{p}{2p-1}}}{1-D_{1}(n, p,\vartheta,R_{0}) \overline{\left\|{\rm{Ric}}_{\infty}^{0}\right\|}_{p, R_{0}, \vartheta}^{\frac{p}{2p-1}}}\nonumber\\
&& +\left(\frac{m\left(B_{R_{i-1}}^{+}(x_{i-1})\right)}{m\left(B_{R_{0}}^{+}(x_{0})\right)}\right)^{\frac{1}{2p-1}}  e^{\frac{\vartheta R_{0}}{2p-1}}\left(\frac{R_{i}}{r_{i}}\right)^{\frac{n}{2p-1}} \nonumber\\
& = & \frac{D_{1}(n, p,\vartheta,R_{0}) e^{\frac{\vartheta R_{0}}{2p-1}} R_{0}^{\frac{n-2p}{2 p-1}} (\Lambda_{F}^{-1}r)^{\frac{2p-n}{2p-1}}}{1-D_{1}(n, p,\vartheta,R_{0}) \overline{\left\|{\rm{Ric}}_{\infty}^{0}\right\|}_{p, R_{0}, \vartheta}^{\frac{p}{2p-1}}}\overline{\left\|{\rm{Ric}}_{\infty}^{0}\right\|}_{p, R_{0}, \vartheta}^{\frac{p}{2p-1}} ((2-\alpha)^{\frac{2p-n}{2p-1}})^{i-1} \nonumber\\
&&+ \left(e^{\vartheta R_{0}}\left(\frac{2-\alpha}{\alpha}\right)^{n} \right)^{\frac{1}{2p-1}} \left(\frac{m\left(B_{R_{i-1}}^{+}(x_{i-1})\right)}{m\left(B_{R_{0}}^{+}(x_{0})\right)}\right)^{\frac{1}{2p-1}}. \label{itera}
\eeq
Put $a_{i}:=\left(\frac{m\left(B_{R_{i}}^{+}(x_{i})\right)}{m\left(B_{R_{0}}^{+}(x_{0})\right)}\right)^{\frac{1}{2 p-1}}$ $(1\leq i \leq k-1)$ and let
\beqn
&& C_{0}:=\frac{D_{1}(n, p,\vartheta,R_{0}) e^{\frac{\vartheta R_{0}}{2p-1}} R_{0}^{\frac{n-2p}{2 p-1}} (\Lambda_{F}^{-1}r)^{\frac{2p-n}{2p-1}}}{1-D_{1}(n, p,\vartheta,R_{0}) \overline{\left\|{\rm{Ric}}_{\infty}^{0}\right\|}_{p, R_{0}, \vartheta}^{\frac{p}{2p-1}}}\overline{\left\|{\rm{Ric}}_{\infty}^{0}\right\|}_{p, R_{0}, \vartheta}^{\frac{p}{2p-1}} \\
&& \beta:=(2-\alpha)^{\frac{2p-n}{2p-1}}, \ \ \ \ \ l:=\left(e^{\vartheta R_{0}}\left(\frac{2-\alpha}{\alpha}\right)^{n} \right)^{\frac{1}{2p-1}}.
\eeqn
Then (\ref{itera}) becomes
$$
a_{i}\leq C_{0}\beta^{i-1}+l a_{i-1}, \quad 2\leq i\leq k-1.
$$
Iterating this inequality gives
$$
a_{i}\leq l^{i-1}\left(a_{1}+C_{0}\sum_{s=1}^{i-1}\left(\frac{\beta}{l}\right)^{s}\right), \quad 2\leq i\leq k-1.
$$
Since $ p\in (\frac{n}{2}, n]$ and $\alpha\in [\frac{1}{2},1)$, we know that $\frac{\beta}{l}=e^{-\frac{\vartheta R_{0}}{2p-1}} \alpha^{\frac{n}{2p-1}}(2-\alpha)^{\frac{2p-2n}{2p-1}}<1$.
Then we can deduce the following
$$
a_{i}\leq   l^{i-1}\Big(a_{1}+\frac{C_{0}}{1-\frac{\beta}{l}}\Big),  \quad 2\leq i\leq k-1.
$$
Taking $i=k-1$ and recalling the definitions of $x_{i}$ and $R_{i}$ give
\beq
&& \left(\frac{m\left(B_{R_{k-1}}^{+}\left(x_{k-1}\right)\right)}{m\left(B_{R_{0}}^{+}(x_{0})\right)}\right)^{\frac{1}{2 p-1}}
\leq \left(e^{\vartheta R_{0}}\left(\frac{2-\alpha}{\alpha}\right)^{n}\right)^{\frac{k-2}{2p-1}} \Big[\left(\frac{m\left(B_{R_{1}}^{+}(x_{1})\right)}{m\left(B_{R_{0}}^{+}(x_{0})\right)}\right)^{\frac{1}{2 p-1}}\nonumber \\
&&+\frac{C_{0}}{1-e^{-\frac{\vartheta R_{0}}{2p-1}}(2-\alpha)^{\frac{2p-2n}{2p-1}} \alpha^{\frac{n}{2 p-1}}}\Big]\nonumber\\
&& = \left(e^{\vartheta R_{0}}\left(\frac{2-\alpha}{\alpha}\right)^{n}\right)^{\frac{k-2}{2p-1}} \left[\left(\frac{m\left(B_{\Lambda_{F}^{-1}r}^{+}(x)\right)}{m\left(B_{R_{0}}^{+}(x_{0})\right)}\right)^{\frac{1}{2 p-1}}+\frac{C_{0}}{1-e^{-\frac{\vartheta R_{0}}{2p-1}}(2-\alpha)^{\frac{2p-2n}{2p-1}} \alpha^{\frac{n}{2 p-1}}}\right]. \label{est1}
\eeq

On the other hand, by {\bf Step 3}, $B_{r_{k}}^{+}(x_{0}) \subset B_{R_{k-1}}^{+}(x_{k-1})$. Then,
\[
\left(\frac{m\left(B_{R_{k-1}}^{+}(x_{k-1})\right)}{m\left(B_{R_{0}}^{+}(x_{0})\right)}\right)^{\frac{1}{2p-1}} \geq \left(\frac{m\left(B_{r_{k}}^{+}(x_{k})\right)}{m\left(B_{R_{0}}^{+}(x_{0})\right)}\right)^{\frac{1}{2p-1}}=\left(\frac{m\left(B_{r_{k}}^{+}(x_{0})\right)}{m\left(B_{R_{0}}^{+}(x_{0})\right)}\right)^{\frac{1}{2p-1}}.
\]
Furthermore, by (\ref{volmv}) with $K=0$, we have
\beqn
&& \left(\frac{m\left(B_{r_{k}}^{+}(x_{0})\right)}{m\left(B_{R_{0}}^{+}(x_{0})\right)}\right)^{\frac{1}{2p-1}} \geq  \left(\frac{v(n,0,r_{k},\vartheta)}{v(n,0,R_{0},\vartheta)}\right)^{\frac{1}{2p-1}}\nonumber\\
&& \times \left(1-C_{2}(n,p,\vartheta,0,R_{0})\left(\frac{v(n,0,R_{0},\vartheta)}{m(B^{+}_{R_{0}}(x_{0}))}\right)^{\frac{1}{2p-1}} \left\|{\rm{Ric}}_{\infty}^{0}\right\|^{\frac{p}{2p-1}}_{p, R_{0}, \vartheta}(x_{0})\right)\nonumber\\
&&\geq \left(\frac{v(n,0,r_{k},\vartheta)}{v(n,0,R_{0},\vartheta)}\right)^{\frac{1}{2p-1}} \left(1-C_{2}(n,p,\vartheta,0,R_{0}) v(n,0,R_{0},\vartheta)^{\frac{1}{2p-1}} \overline{\left\|{\rm{Ric}}_{\infty}^{0}\right\|}_{p, R_{0}, \vartheta}^{\frac{p}{2p-1}}\right)\nonumber\\
&& \geq e^{-\frac{\vartheta R_{0}}{2p-1}}\left(\frac{r_{k}}{R_{0}}\right)^{\frac{n}{2 p-1}} \left(1-D_{1}(n, p, \vartheta , R_{0})\overline{\left\|{\rm{Ric}}_{\infty}^{0}\right\|}_{p, R_{0}, \vartheta}^{\frac{p}{2 p-1}}\right)\nonumber\\
&& = e^{-\frac{\vartheta R_{0}}{2 p-1}}\Lambda_{F}^{-\frac{n}{2p-1}} \alpha^{\frac{n}{2 p-1}} (2-\alpha)^{\frac{(k-2)n}{2p-1}}\left(\frac{r}{R_{0}}\right)^{\frac{n}{2p-1}} \left(1-D_{1}(n,p,\vartheta, R_{0}) \overline{\left\|{\rm{Ric}}_{\infty}^{0}\right\|}_{p, R_{0}, \vartheta}^{\frac{p}{2 p-1}}\right),
\eeqn
where we have used (\ref{iwRp}) and (\ref{iwRM}) in second inequality and used (\ref{v0}), (\ref{C2v}) in third inequality. We have also used the definition of $r_{k}$ in the last equality. Then we conclude the following
\beq
\left(\frac{m\left(B_{R_{k-1}}^{+}(x_{k-1})\right)}{m\left(B_{R_{0}}^{+}(x_{0})\right)}\right)^{\frac{1}{2p-1}}&\geq & e^{-\frac{\vartheta R_{0}}{2 p-1}}\Lambda_{F}^{-\frac{n}{2p-1}} \alpha^{\frac{n}{2 p-1}} (2-\alpha)^{\frac{(k-2)n}{2p-1}}\left(\frac{r}{R_{0}}\right)^{\frac{n}{2p-1}} \nonumber\\
&& \times \left(1-D_{1}(n,p,\vartheta, R_{0}) \overline{\left\|{\rm{Ric}}_{\infty}^{0}\right\|}_{p, R_{0}, \vartheta}^{\frac{p}{2 p-1}}\right). \label{est2}
\eeq
Thus, combining (\ref{est1}) with (\ref{est2}) yields
\beq
\left(\frac{m\left(B_{\Lambda_{F}^{-1}r}^{+}(x)\right)}{m\left(B_{R_{0}}^{+}(x_{0})\right)}\right)^{\frac{1}{2 p-1}}
&\geq & e^{-\frac{\vartheta R_{0}}{2p-1}(k-1)} \Lambda_{F}^{-\frac{n}{2p-1}} \alpha^{\frac{(k-1)n}{2p-1}} \left(\frac{r}{R_{0}}\right)^{\frac{n}{2p-1}} \left(1-D_{1}(n,p,\vartheta, R_{0}) \overline{\left\|{\rm{Ric}}_{\infty}^{0}\right\|}_{p, R_{0}, \vartheta}^{\frac{p}{2 p-1}}\right) \nonumber\\
& -& \frac{D_{1}(n, p,\vartheta,R_{0}) e^{\frac{\vartheta R_{0}}{2p-1}} R_{0}^{\frac{n-2p}{2 p-1}} (\Lambda_{F}^{-1}r)^{\frac{2p-n}{2p-1}} \overline{\left\|{\rm{Ric}}_{\infty}^{0}\right\|}_{p, R_{0}, \vartheta}^{\frac{p}{2p-1}}}{\left(1-D_{1}(n, p,\vartheta,R_{0}) \overline{\left\|{\rm{Ric}}_{\infty}^{0}\right\|}_{p, R_{0}, \vartheta}^{\frac{p}{2p-1}}\right)\left(1-e^{-\frac{\vartheta R_{0}}{2p-1}} \alpha^{\frac{n}{2 p-1}}(2-\alpha)^{\frac{2p-2n}{2p-1}}\right)}. \label{nonconcom}
\eeq
Moreover, since $ p\in (\frac{n}{2}, n]$ and $\alpha\in [\frac{1}{2},1)$, it is not difficult to see that $\frac{n}{2p-1}\leq \frac{n}{n-1}\leq 2$ and $-\ln \alpha\leq 2\ln(2-\alpha)$.
By (\ref{floor}) and noticing that $\ln \alpha <0$, we have
\beqn
\alpha^{\frac{(k-1)n}{2p-1}}
&=& \left(e^{(k-2)\ln \alpha}\cdot\alpha\right)^{\frac{n}{2p-1}}
\geq \left(e^{\ln \Big(1+\tfrac{d_{F}(x_{0}, x)}{\Lambda_{F}^{-2} r}\Big)\frac{\ln \alpha}{\ln (2-\alpha)}}\cdot\alpha\right)^{\frac{n}{2p-1}}\\
&=& \left(\Big(1+\frac{d_{F}(x_{0}, x)}{\Lambda_{F}^{-2} r}\Big)^{\frac{\ln \alpha}{\ln (2-\alpha)}}\cdot\alpha\right)^{\frac{n}{2p-1}}
\geq \alpha^{\frac{n}{2p-1}} \left(\frac{\Lambda_{F}^{-2} r}{R_{0}}\right)^{\frac{2 n}{2p-1}} \geq \frac{1}{4}\Lambda_{F}^{-\frac{4n}{2p-1}}\big(\frac{r}{R_{0}}\big)^{\frac{2n}{2p-1}}
\eeqn
and
\beqn
e^{-\frac{(k-1)\vartheta R_{0}}{2p-1}}&=&(e^{k-2}\cdot e)^{-\frac{\vartheta R_{0}}{2p-1}}
\geq \left(\left(1+\frac{d_{F}(x_{0}, x)}{\Lambda_{F}^{-2} r}\right)^{\frac{1}{\ln (2-\alpha)}}\cdot e\right)^{-\frac{\vartheta R_{0}}{2p-1}}\\
&\geq& e^{-\frac{\vartheta R_{0}}{2p-1}} \left(\frac{R_{0}}{\Lambda_{F}^{-2} r}\right)^{-\vartheta R_{0} b(p,\alpha)}
=e^{-\frac{\vartheta R_{0}}{2p-1}}\Lambda_{F}^{-2\vartheta R_{0} b(p,\alpha)} \left(\frac{r}{R_{0}}\right)^{\vartheta R_{0} b(p,\alpha)},
\eeqn
where $b(p,\alpha)=\frac{1}{(2p-1)\ln(2-\alpha)}$. Finally, from (\ref{nonconcom}) and these two inequalities, we obtain (\ref{nconcom}).
\end{proof}
\vskip 2mm

As the end of this section, we give the following lemma, which will be used to prove the compactness of the universal covering space $\widetilde{M}$ of $M$ in Theorem \ref{boma}.

\begin{lem}\label{chfen}{\rm (\cite{ChF1}, Theorem 1.1)} Let $(M, F, m)$ be an $n$-dimensional forward complete Finsler measure space. Assume that ${\rm Ric}_{\infty}\geq -K$  for some $K\geq 0$.  Then, for any $0< r_{1} < r_{2}$, we have
\be
\frac{m(B_{r_{2}}(x_0))}{m(B_{r_{1}}(x_0))}\leq \left(\frac{r_{2}}{r_{1}}\right)^{n+1}e^{\frac{K+ \vartheta_{0}^{2}}{6}r_{2}^2}, \label{doubvol}
\ee
where  $x_{0}\in M$ is an arbitrary fixed point and $\vartheta_{0}$ is defined by (\ref{supS}).
\end{lem}

\section{A theorem of Bonnet-Myers type} \label{bon-May}
In this section, we will derive a theorem of Bonnet-Myers type on Finsler metric measure manifolds with integral weighted Ricci curvature bounds. First of all, we give the following fundamental theorem.

\begin{thm}\label{myers}
Let $(M, F, m)$ be an $n$-dimensional $(n \geq 2)$  forward complete connected Finsler metric measure manifold with $\Lambda_{F}<\infty$. Assume that $\mathbf{S}\geq -\vartheta$ for some $\vartheta \geq 0$. Let $R_{0}>\Lambda_{F}\pi$ be a constant. Given $p \in (\frac{n}{2}, n]$, $\alpha\in [\frac{1}{2},1)$. Then there exist nonnegative constant $\varepsilon =\varepsilon(n,p,\Lambda_{F},\vartheta,\alpha,R_{0})$ and positive constants $C_{3}=C_{3}(n,p,\Lambda_{F},\vartheta,\alpha,R_{0})$ and $\beta =\beta(n,p,\vartheta,\alpha,R_{0})$ such that when
$$
\overline{\left\|{\rm{Ric}}_{\infty}^{\Lambda_{F}^{2}}\right\|}_{p, R_{0}, \vartheta}\leq \varepsilon ,
$$
we have
\be
{\rm{diam}}(M) \leq \Lambda_{F}^{-1}\left(\pi+C_{3}\overline{\left\|{\rm{Ric}}_{\infty}^{\Lambda_{F}^{2}}\right\|}_{p, R_{0}, \vartheta}^{\beta}\right) \leq \frac{\Lambda_{F}^{-1}(\Lambda_{F}+1)\pi}{2}, \label{diam}
\ee
where
\beqn
\varepsilon &:=&\min\Big\{\left(\frac{\pi}{6}\right)^{2-\frac{1}{p}}, C_{3}^{-\frac{1}{\beta}}\left(\frac{(\Lambda_{F}-1)\pi}{2}\right)^{\frac{1}{\beta}}, \left(\frac{1}{2}\right)^{\frac{2p-1}{p}} D_{1}(n,p,\vartheta,R_{0})^{-\frac{2p-1}{p}}\Big\},\\
C_{3} &:=&\max\Big\{\left(32 e^{\frac{\vartheta R_{0}}{2 p-1}} \Lambda_{F}^{\left(3\vartheta R_{0} b(p,\alpha)+\frac{9n-2p}{2p-1}\right)} D_{2}\right)^{\frac{2p-1}{(2p-1)\vartheta R_{0} b(p,\alpha)+4n-2p}} R_{0},\\
&&\ \ \ \ \ \ \left(2^{8p-3}C_{1}e^{2\vartheta R_{0}}\Lambda_{F}^{3(2p-1)\vartheta R_{0} b(p,\alpha)+8n-1} \right)^{\frac{1}{(2p-1)\vartheta R_{0}b(p,\alpha)+3n-1}} R_{0}^{2},\\
&&\ \ \ \ \ \ \left(2\Lambda_{F}^{-\frac{2p-n}{2p-1}}D_{2}(n,p,\Lambda_{F},\vartheta,\alpha,R_{0})\right)^{-\frac{2p-1}{2p-n}} R_{0}\Big\},\\
\beta &:=& \min\left\{\frac{p}{(2p-1)\vartheta R_{0} b(p,\alpha)+4n-2p}, \ \frac{(n-1)p}{(2p-1)\left[(2p-1)\vartheta R_{0}b(p,\alpha)+3n-1\right]}\right\}.
\eeqn
Furthermore, $M$ is in fact compact.
\end{thm}
\begin{proof}
We will give the proof of (\ref{diam}) by contradiction. Suppose that for constants $\varepsilon$, $C_{3}$ and $\beta$ defined as above, when $\overline{\left\|{\rm{Ric}}_{\infty}^{\Lambda_{F}^{2}}\right\|}_{p, R_{0}, \vartheta}\leq \varepsilon$, there exist $x_{0}$, $x \in M$ such that
\beq\label{d(x0,x)}
d_{F}(x_{0}, x) > \Lambda_{F}^{-1}\left(\pi+C_{3}\overline{\left\|{\rm{Ric}}_{\infty}^{\Lambda_{F}^{2}}\right\|}_{p, R_{0}, \vartheta}^{\beta}\right).
\eeq
For convenience, define $\xi:=\Lambda_{F}^{-1}C_{3}\overline{\left\|{\rm{Ric}}_{\infty}^{\Lambda_{F}^{2}}\right\|}_{p, R_{0}, \vartheta}^{\beta}$. By the assumption and the definition of $\varepsilon$, it is not difficult to see that $\xi<\frac{\Lambda_{F}^{-1}(R_{0}-\pi)}{2}$. Besides, from (\ref{d(x0,x)}), we can see that $\xi < d_{F}(x_{0}, x)-\Lambda_{F}^{-1}\pi$. Then we can choose a constant $\delta$ such that $\delta\in \left(\xi, \ \min\left\{\frac{\Lambda_{F}^{-1}(R_{0}-\pi)}{2}, \ d_{F}(x_{0},x)-\Lambda_{F}^{-1}\pi\right\}\right)$.
In the following, we will derive contradictions from (\ref{d(x0,x)}).

Since $d_{F}(x_{0},x)>\Lambda_{F}^{-1}\pi +\delta$ and $M$ is connected, there exists $x^{\prime}\in M$ such that $d_{F}(x_{0},x^{\prime})=\Lambda_{F}^{-1}\pi+\delta$. Further,  it is easy to show that
\be
B^{+}_{\Lambda_{F}^{-1}\delta}(x^{\prime})\subset B^{+}_{\Lambda_{F}^{-1}(\pi+\delta)+\delta}(x_{0})\backslash B^{+}_{\Lambda_{F}^{-1}\pi}(x_{0}) \subset B_{\Lambda_{F}^{-1}\pi+2\delta}^{+}(x_{0})\backslash B^{+}_{\Lambda_{F}^{-1}\pi}(x_{0}). \label{ball0}
\ee
By $2\delta <\Lambda_{F}^{-1}(R_{0}-\pi)$, we can see that  $\Lambda_{F}^{-1}\pi+2\delta < \Lambda_{F}^{-1}R_{0}$. Moreover, by the assumption, we have $\overline{\left\|{\rm{Ric}}_{\infty}^{\Lambda_{F}^{2}}\right\|}_{p, R_{0}, \vartheta}(x_{0}) \leq \varepsilon \leq \left(\frac{\pi}{6}\right)^{2-\frac{1}{p}}$. Hence, from (\ref{ball0}) and Lemma \ref{S(r)}, we have
\be\label{vol3}
m(B^{+}_{\Lambda_{F}^{-1}\delta}(x^{\prime})) \leq \int_{\Lambda_{F}^{-1}\pi}^{\Lambda_{F}^{-1}\pi+2\delta } A(r)dr \leq  2 \delta C_{1} e^{\vartheta R_{0}} \overline{\left\|{\rm{Ric}}_{\infty}^{\Lambda_{F}^{2}}\right\|}_{p, R_{0}, \vartheta}^{\frac{(n-1)p}{2p-1}}(x_{0}) \, m(B_{R_{0}}^{+}(x_{0})),
\ee
where $C_{1}=C_{1}(n,p,\Lambda_{F},R_{0})$ is the constant given in Lemma \ref{S(r)}.

On the other hand,  observe that $\Lambda_{F} d_{F}(x_{0}, x^{\prime})+\delta= \pi+(\Lambda_{F}+1)\delta  \leq \pi +2\Lambda_{F}\delta  \leq R_{0}$, that is, $d_{F}(x_{0}, x^{\prime})\leq \Lambda_{F}^{-1}(R_{0}-\delta)$.  It follows from Theorem \ref{noncon} that
\beqn
\frac{m\left(B_{\Lambda_{F}^{-1}\delta}^{+}(x^{\prime})\right)}{m\left(B_{R_{0}}^{+}(x_{0})\right)}
&\geq & \Big[\frac{1}{4}\left(1-D_{1}\overline{\left\|{\rm{Ric}}_{\infty}^{0}\right\|}_{p, R_{0}, \vartheta}^{\frac{p}{2p-1}}\right) e^{-\frac{\vartheta R_{0}}{2 p-1}}\Lambda_{F}^{-\left(2 \vartheta R_{0}b(p,\alpha)+\frac{5n}{2p-1}\right)} \left(\frac{\delta}{R_{0}}\right)^{\left(\vartheta R_{0} b(p,\alpha)+\frac{4n-2p}{2p-1}\right)} \\
&& -\frac{D_{2} \overline{\left\|{\rm{Ric}}_{\infty}^{0}\right\|}_{p, R_{0}, \vartheta}^{\frac{p}{2p-1}}}{1-D_{1} \overline{\left\|{\rm{Ric}}_{\infty}^{0}\right\|}_{p, R_{0}, \vartheta}^{\frac{p}{2p-1}}} \Big]^{2p-1} \left(\frac{\delta}{R_{0}}\right)^{2p-n}.
\eeqn
Since $\overline{\left\|{\rm{Ric}}_{\infty}^{0}\right\|}_{p, R_{0}, \vartheta}\leq\overline{\left\|{\rm{Ric}}_{\infty}^{\Lambda_{F}^{2}}\right\|}_{p, R_{0}, \vartheta}\leq \varepsilon \leq \left(\frac{1}{2}\right)^{\frac{2p-1}{p}} D_{1}^{-\frac{2p-1}{p}}$,
we have
$$
1-D_{1}\overline{\left\|{\rm{Ric}}_{\infty}^{0}\right\|}_{p, R_{0}, \vartheta}^{\frac{p}{2p-1}}\geq \frac{1}{2}.
$$
Hence
\beq
&&\frac{m\left(B^{+}_{\Lambda_{F}^{-1}\delta}(x^{\prime})\right)}{m(B^{+}_{R_{0}}(x_{0}))}
\geq  \left[\frac{1}{8}\right. e^{-\frac{\vartheta R_{0}}{2 p-1}}\Lambda_{F}^{-\left(2 \vartheta R_{0}b(p,\alpha)+\frac{5n}{2p-1}\right)} \left(\frac{\delta}{R_{0}}\right)^{\vartheta R_{0} b(p,\alpha)+\frac{4n-2p}{2p-1}}\nonumber \\
&& \ \ \ \ \ \ \ \ \ \ \ \ \ \ \ \ \ \ \ \ \ \left.- 2D_{2}\overline{\left\|{\rm{Ric}}_{\infty}^{0}\right\|}_{p, R_{0}, \vartheta}^{\frac{p}{2p-1}}\right]^{2 p-1} \left(\frac{\delta}{R_{0}}\right)^{2p-n} \nonumber\\
&&\geq \left[\frac{1}{8} e^{-\frac{\vartheta R_{0}}{2 p-1}}\Lambda_{F}^{-\left(2 \vartheta R_{0}b(p,\alpha)+\frac{5n}{2p-1}\right)} \left(\frac{\delta}{R_{0}}\right)^{\vartheta R_{0} b(p,\alpha)+\frac{4n-2p}{2p-1}}- 2 D_{2}\overline{\left\|{\rm{Ric}}_{\infty}^{\Lambda_{F}^{2}}\right\|}_{p, R_{0}, \vartheta}^{\frac{p}{2p-1}}\right]^{2 p-1} \left(\frac{\delta}{R_{0}}\right)^{2p-n}.  \label{ball3}
\eeq

Now, by the definitions of $C_{3}, \ \beta$ and $\xi$, we will derive contradictions by two cases.
\ben
\item[{\rm (1)}] Assume that
$$
e^{-\frac{\vartheta R_{0}}{2 p-1}}\Lambda_{F}^{-\left(2\vartheta R_{0} b(p,\alpha)+\frac{5n}{2p-1}\right)} \left(\frac{\delta}{R_{0}}\right)^{\vartheta R_{0} b(p,\alpha)+\frac{4n-2p}{2p-1}} \leq 32 D_{2} \overline{\left\|{\rm{Ric}}_{\infty}^{\Lambda_{F}^{2}}\right\|}_{p, R_{0}, \vartheta}^{\frac{p}{2p-1}}.
$$
In this case, we can get
\beqn
\delta
&\leq& \Lambda_{F}^{-1}\left(32 e^{\frac{\vartheta R_{0}}{2 p-1}} \Lambda_{F}^{\left(3\vartheta R_{0} b(p,\alpha)+\frac{9n-2p}{2p-1}\right)} D_{2}\right)^{\frac{2p-1}{(2p-1)\vartheta R_{0} b(p,\alpha)+4n-2p}}R_{0} \overline{\left\|{\rm{Ric}}_{\infty}^{\Lambda_{F}^{2}}\right\|}_{p, R_{0}, \vartheta}^{\frac{p}{(2p-1)\vartheta R_{0} b(p,\alpha)+4n-2p}}\\
&\leq & \Lambda_{F}^{-1} C_{3} \overline{\left\|{\rm{Ric}}_{\infty}^{\Lambda_{F}^{2}}\right\|}_{p, R_{0}, \vartheta}^{\beta}= \xi.
\eeqn
This contradicts the choice of $\delta$.

\item[{\rm (2)}] Assume that
$$
e^{-\frac{\vartheta R_{0}}{2 p-1}}\Lambda_{F}^{-\left(2\vartheta R_{0} b(p,\alpha)+\frac{5n}{2p-1}\right)} \left(\frac{\delta}{R_{0}}\right)^{\vartheta R_{0} b(p,\alpha)+\frac{4n-2p}{2p-1}}
> 32D_{2} \overline{\left\|{\rm{Ric}}_{\infty}^{\Lambda_{F}^{2}}\right\|}_{p, R_{0}, \vartheta}^{\frac{p}{2p-1}}.
$$
In this case, (\ref{ball3}) implies that
\be
\frac{m\left(B^{+}_{\Lambda_{F}^{-1}\delta}(x^{\prime})\right)}{m(B^{+}_{R_{0}}(x_{0}))}
\geq \left(\frac{1}{16}e^{-\frac{\vartheta R_{0}}{2 p-1}}\Lambda_{F}^{-\left(2\vartheta R_{0} b(p,\alpha)+\frac{5n}{2p-1}\right)}\right)^{2p-1}  \left(\frac{\delta}{R_{0}}\right)^{(2p-1)\vartheta R_{0} b(p,\alpha)+3n}.  \label{volm}
\ee
Then combining (\ref{volm}) with (\ref{vol3}) yields
\beq
\delta &<& \Lambda_{F}^{-1} \left(2^{8p-3}C_{1}e^{2\vartheta R_{0}}\Lambda_{F}^{3(2p-1)\vartheta R_{0}b(p,\alpha)+8n-1} \right)^{\frac{1}{(2p-1)\vartheta R_{0}b(p,\alpha)+3n-1}} R_{0}^{\frac{(2p-1)\vartheta R_{0}b(p,\alpha)+3n}{(2p-1)\vartheta R_{0}b(p,\alpha)+3n-1}} \nonumber\\
&& \times\overline{\left\|{\rm{Ric}}_{\infty}^{\Lambda_{F}^{2}}\right\|}_{p, R_{0}, \vartheta}^{\frac{(n-1)p}{(2p-1)[(2p-1)\vartheta R_{0}b(p,\alpha)+3n-1]}}\nonumber\\
&<& \Lambda_{F}^{-1} \left(2^{8p-3}C_{1}e^{2\vartheta R_{0}}\Lambda_{F}^{3(2p-1)\vartheta R_{0}b(p,\alpha)+8n-1} \right)^{\frac{1}{(2p-1)\vartheta R_{0}b(p,\alpha)+3n-1}} R_{0}^{2} \ \overline{\left\|{\rm{Ric}}_{\infty}^{\Lambda_{F}^{2}}\right\|}_{p, R_{0}, \vartheta}^{\beta}  \nonumber\\
&\leq & \Lambda_{F}^{-1} C_{3} \overline{\left\|{\rm{Ric}}_{\infty}^{\Lambda_{F}^{2}}\right\|}_{p, R_{0}, \vartheta}^{\beta}= \xi. \label{delta1}
\eeq
This also contradicts the choice of $\delta$.
\een

In summary, we have proved that, when $\overline{\left\|{\rm{Ric}}_{\infty}^{\Lambda_{F}^{2}}\right\|}_{p, R_{0}, \vartheta}\leq \varepsilon$, one has
$$
d_{F}(x_{0},x) \leq \Lambda_{F}^{-1}\left(\pi+C_{3} \overline{\left\|{\rm{Ric}}_{\infty}^{\Lambda_{F}^{2}}\right\|}_{p, R_{0}, \vartheta}^{\beta}\right)\leq \frac{\Lambda_{F}^{-1}(\Lambda_{F}+1)\pi}{2}
$$
for any $x_{0}$ and $x\in M$. This proves (\ref{diam}).

Further, since $(M, F, m)$ is forward complete and  we have just shown that it is forwardly bounded from the above, $M$ is compact from the Hopf-Rinow theorem (\cite{BaoChernShen}, Theorem 6.6.1).  This completes the proof.
\end{proof}

\vskip 2mm

For a Finsler metric measure manifold $(M, F, m)$, let $\widetilde{M}$ denote a covering space of $M$ and let $f: \widetilde{M} \rightarrow M$ be the covering mapping with deck transformation group $\Gamma$. $\Omega \subset \widetilde{M}$ is called a fundamental domain of $\widetilde{M}$ if $f(\bar{\Omega})=M$ and $\rho(\Omega) \cap \Omega=\emptyset$ for all $\rho \in \Gamma-\{1\}$. If $\Omega$ is a fundamental domain, then
$$
\bigcup_{\rho \in \Gamma} \rho(\bar{\Omega})=\widetilde{M},\left.\quad f\right|_{\rho(\Omega)}: \rho(\Omega) \rightarrow f(\Omega) \text{ is a homeomorphism for any}  \ \rho \in \Gamma.
$$
If $(M, F)$ is forward complete, we can always get a fundamental domain of $\widetilde{M}$ (see \cite{Grov, Zhao}).

For any $\tilde{x} \in \widetilde{M}$, let $x= f(\tilde{x})\in M$. Then there exists an admissible open set $U \subseteq M$ containing $x$ such that $f^{-1}(U)$ is a disjoint union of open sets in $\widetilde{M}$, each of which is mapped homeomorphically onto $U$ by $f$. Then $\tilde{x}$ must lie in one of these open sets, denoted by  $\widetilde{U}$. Without loss of generality, assume that $U$ is a local coordinate neighborhood  of $x$ with homeomorphism $\phi : U \rightarrow \phi (U)\subset \mathbb{R}^{n}$. Thus  $\widetilde{U}$ is a local coordinate neighborhood  of $\tilde{x}$ with homeomorphism $\phi \circ f: \widetilde{U} \rightarrow \phi (U)\subset \mathbb{R}^{n}$. Naturally, define the pull-back  $\widetilde{F}:= f^{*}F$ of $F$ by $\widetilde{F}(\tilde{x},\tilde{y}):= F(f(\tilde{x}),df_{\tilde{x}}(\tilde{y}))$ for all $\tilde{x}\in\widetilde{M}$, $\tilde{y}\in T_{\tilde{x}}\widetilde{M}\setminus\{0\}$, where $df_{\tilde{x}}: T_{\tilde{x}}\widetilde{M}\to T_{x}M$ is the linear isomorphism induced by $f$ between the tangent spaces. Further, equip a pull-back measure $\widetilde{m}$ for $\widetilde{M}$ satisfying  $d\widetilde{m}|_{\widetilde{U}}:= (f|_{\widetilde{U}})^{*}(dm|_{U})$ for all such $\widetilde{U}$ and $U$. Then $(\widetilde{M}, \widetilde{F}, \widetilde{m})$ is a Finsler metric measure manifold. In the following, we adopt $\widetilde{\cdot}$ to denote the corresponding quantities on Finsler manifold $(\widetilde{M}, \widetilde{F}, \widetilde{m})$. Based on the discussions as above,
we know that $\widetilde{\mathbf{S}}(\tilde{x},\tilde{y})=\mathbf{S}(f(\tilde{x}),df_{\tilde{x}}(\tilde{y}))$,  $\widetilde{\rm{Ric}}(\tilde{x},\tilde{y})={\rm{Ric}}(f(\tilde{x}),df_{\tilde{x}}(\tilde{y}))$,   $\widetilde{\rm{Ric}}_{\infty}(\tilde{x},\tilde{y})={\rm{Ric}}_{\infty}(f(\tilde{x}),df_{\tilde{x}}(\tilde{y}))$.

Now we are in the position to prove the following theorem.

\begin{thm}\label{boma}
Let $(M, F, m)$ be an $n$-dimensional $(n\geq 2)$  forward complete connected Finsler metric measure manifold with $\Lambda_{F}<\infty$. Assume that $\mathbf{S}\geq 0$. Let $R_{0}>0$ be a constant.
Given $p\in (\frac{n}{2}, n]$, $\alpha\in [\frac{1}{2},1)$. Then we can get the following results.
\ben
\item[{\rm (1)}] There exists a nonnegative constant $\varepsilon_{0}=\varepsilon_{0}(n,p,\Lambda_{F},\alpha,R_{0})$ such that when $\overline{\left\|{\rm{Ric}}_{\infty}^{\Lambda_{F}^{2}}\right\|}_{p, R_{0}, 0}\leq \varepsilon_{0}$,
we have
\be
{\rm{diam}}(M)\leq \frac{\Lambda_{F}^{-1}(\Lambda_{F}+1)\pi}{2}. \label{diaest-2}
\ee
In particular, $M$ is compact.

\item[{\rm (2)}]  If $(M, F, m)$ further satisfies ${\rm Ric}_{\infty}\geq -K$  for some $K\geq 0$, there exists a nonnegative constant $\bar{\varepsilon}_{0}= \bar{\varepsilon}_{0}(n,p,\Lambda_{F}, \alpha, R_{0}, K, \vartheta_{0})$ such that when $\overline{\left\|{\rm{Ric}}_{\infty}^{\Lambda_{F}^{2}}\right\|}_{p, R_{0}, 0}\leq \bar{\varepsilon}_{0}$,  $\pi_{1}(M)$ is finite.
\een
\end{thm}
\begin{proof} We will give the proof by two steps.

$\mathbf{Step\ 1.}$ \ Firstly, we derive the diameter bound of $M$ for any $R_{0}>0$.

If $R_{0}>\Lambda_{F}\pi$, the conclusion follows directly from Theorem \ref{myers}.

In the following, we will prove (\ref{diaest-2}) in the case $R_{0}\leq \Lambda_{F}\pi$.

Fix a constant $\bar{R}>\Lambda_{F}\pi$ and choose arbitrarily a point $x_{0}\in M$. Put
\[
{\cal B}:=\left\{\left.B_{\frac{R_0}{5 \Lambda_F}}^{+}(z) \right\rvert\, z \in B_{\bar{R}}^{+}(x_0 ) \ \text{and} \ B_{\frac{R_0}{5 \Lambda_F}}^{+}(z) \subset B_{\bar{R}}^{+}(x_0 )\right\}.
\]
It is obvious that ${\cal B}$ is nonempty. Further, let
$$
{\cal P}:=\left\{{\cal A} \subset {\cal B} \mid    B_{1} \cap B_{2} = \emptyset \ \text{for any} \  B_{1} \neq B_{2}\in {\cal A}\right\}.
$$
Then ${\cal P}$ is a nonempty partially ordered set and  every chain in ${\cal P}$ has upper bound.  By Zorn's Lemma, ${\cal P}$ has a maximal element denoted by
\[
{\cal A}_{*}= \Big\{B^{+}_{\frac{R_{0}}{5\Lambda_{F}}}(x_{i})\Big\}_{i\in I}.
\]
Here $I$ denotes the index set. ${\cal A}_{*}$ is a maximal pairwise disjoint family contained in $B^{+}_{\bar{R}}(x_{0})$.

Further, we claim that $\left\{B^{+}_{R_{0}}(x_{i})\right\}_{i\in I}$ actually covers $B^{+}_{\bar{R}}(x_{0})$.
Indeed, suppose that $\left\{B^{+}_{R_{0}}(x_{i})\right\}_{i\in I}$ is not a covering of $B^{+}_{\bar{R}}(x_{0})$. Then there exists a point $\bar{x}\in B^{+}_{\bar{R}}(x_{0})$ such that
$$\bar{x}\notin \bigcup\limits_{i\in I} B^{+}_{R_{0}}(x_{i}).$$
By the maximality of $\Big\{B^{+}_{\frac{R_{0}}{5\Lambda_{F}}}(x_{i})\Big\}_{i\in I}$, there exists some  $i_{0}\in I$ such that
$$B^{+}_{\frac{R_{0}}{5\Lambda_{F}}}(\bar{x})\cap B^{+}_{\frac{R_{0}}{5\Lambda_{F}}}(x_{i_{0}})\neq \emptyset.$$
Choose $z_{0}\in B^{+}_{\frac{R_{0}}{5\Lambda_{F}}}(\bar{x})\cap B^{+}_{\frac{R_{0}}{5\Lambda_{F}}}(x_{i_{0}})$. Then for any $z\in B^{+}_{\frac{R_{0}}{5\Lambda_{F}}}(\bar{x})$, we have
\beqn
d_{F}(x_{i_{0}},z) &\leq & d_{F}(x_{i_{0}},z_{0})+d_{F}(z_{0},\bar{x})+d_{F}(\bar{x},z) \\
&\leq & d_{F}(x_{i_{0}},z_{0})+\Lambda_{F}\cdot d_{F}(\bar{x},z_{0})+d_{F}(\bar{x},z)< R_{0}.
\eeqn
Thus, $B^{+}_{\frac{R_{0}}{5\Lambda_{F}}}(\bar{x})\subset B^{+}_{R_{0}}(x_{i_{0}})$, which contradicts the choice of $\bar{x}$. Therefore, the family $\left\{B^{+}_{R_{0}}(x_{i})\right\}_{i\in I}$ is a covering of $B^{+}_{\bar{R}}(x_{0})$, from which we have
\beq\label{est3}
\overline{\left\|{\rm{Ric}}^{\Lambda_{F}^{2}}_{\infty}\right\|}^{p}_{p,\bar{R},0}(x_{0})
&=&\frac{1}{m(B^{+}_{\bar{R}}(x_{0}))} \int^{\bar{R}}_{0}\int_{\mathcal{D}_{x_0}(r)}\left({\rm{Ric}}_{\infty}^{\Lambda_{F}^{2}}\right)^p \sigma(x_0, r, \theta) d r d \theta \nonumber\\
& \leq& \frac{1}{m(B^{+}_{\bar{R}}(x_{0}))} \sum_{i \in I} \int_0^{R_{0}} \int_{\mathcal{D}_{x_i}(r)}\left({\rm{Ric}}_{\infty}^{\Lambda_{F}^{2}}\right)^p \sigma\left(x_{i}, r, \theta\right) d r d \theta \nonumber\\
& =&\sum_{i \in I} \frac{m\left(B_{R_{0}}^{+}(x_i)\right)}{m\left(B_{\bar{R}}^{+}(x_0)\right)} \cdot \frac{1}{m\left(B_{R_{0}}^{+}(x_i)\right)} \int_0^{R_{0}} \int_{\mathcal{D}_{x_i}(r)}\left({\rm{Ric}}_{\infty}^{\Lambda_{F}^{2}}\right)^p \sigma\left(x_i, r, \theta\right) d r d \theta\nonumber\\
& =& \sum_{i \in I} \frac{m\left(B_{R_{0}}^{+}(x_i)\right)}{m\left(B_{\bar{R}}^{+}(x_0)\right)} \overline{\left\|{\rm{Ric}}^{\Lambda_{F}^{2}}_{\infty}\right\|}^{p}_{p,R_{0},0}(x_{i})\nonumber\\
& =&\sum_{i \in I} \frac{m\left(B_{R_{0}}^{+}(x_i)\right)} {m\Big(B_{\frac{R_{0}}{5\Lambda_{F}}}^{+}(x_i)\Big)} \cdot \frac{m\Big(B_{\frac{R_{0}}{5\Lambda_{F}}}^{+}(x_i)\Big)}{m\left(B_{\bar{R}}^{+}(x_0)\right)} \overline{\left\|{\rm{Ric}}^{\Lambda_{F}^{2}}_{\infty}\right\|}^{p}_{p,R_{0},0}(x_{i}).
\eeq

Now, let $\varepsilon_{1} =\varepsilon_{1}(n,p,R_{0})$ be the constant $\varsigma$ in Theorem \ref{BG} with $\Xi=5$.
Then, when $\overline{\left\|{\rm{Ric}}_{\infty}^{\Lambda_{F}^{2}}\right\|}_{p, R_{0}, 0}\leq \varepsilon_{1}$, we obtain
$$
\frac{m\left(B_{R_{0}}^{+}(x_i)\right)}{m\Big(B_{\frac{R_{0}}{5\Lambda_{F}}}^{+}(x_i)\Big)} \leq 5^{n+1}\Lambda_{F}^{n}.
$$
Substituting this into (\ref{est3}) yields
$$
\overline{\left\|{\rm{Ric}}_{\infty}^{\Lambda_{F}^{2}}\right\|}^p_{p, \bar{R}, 0} (x_{0})
\leq 5^{n+1} \Lambda_{F}^{n} \sum_{i \in I} \frac{m\Big(B_{\frac{R_0}{5\Lambda_{F}}}^{+}(x_{i})\Big)}{m\left(B_{\bar{R}}^{+}(x_{0})\right)} \ \overline{\left\|{\rm{Ric}}^{\Lambda_{F}^{2}}_{\infty}\right\|}_{p, R_{0}, 0}^p(x_{i})
\leq 5^{n+1} \Lambda_{F}^{n} \ \overline{\left\|{\rm{Ric}}^{\Lambda_{F}^{2}}_{\infty}\right\|}_{p, R_{0}, 0}^{p},
$$
where we have used the fact that $\sum\limits_{i\in I} m\Big(B^{+}_{\frac{R_0}{5\Lambda_{F}}}(x_{i})\Big) \leq m(B^{+}_{\bar{R}}(x_{0}))$ and (\ref{iwRM}) in the last inequality.
Thus, we have
\be\label{ric1}
\overline{\left\|{\rm{Ric}}_{\infty}^{\Lambda_{F}^{2}}\right\|}_{p, \bar{R}, 0}=\sup_{x\in M} \overline{\left\|{\rm{Ric}}_{\infty}^{\Lambda_{F}^{2}}\right\|}_{p, \bar{R}, 0}(x)
\leq 5^{\frac{n+1}{p}} \Lambda_{F}^{\frac{n}{p}} \overline{\left\|{\rm{Ric}}^{\Lambda_{F}^{2}}_{\infty}\right\|}_{p, R_{0}, 0}.
\ee

Next, choose $\bar{R}:=2\Lambda_{F}\pi$ and
$$
\varepsilon_{2}(n,p,\Lambda_{F},\alpha):= \frac{\varepsilon}{5^{\frac{n+1}{p}} \Lambda_{F}^{\frac{n}{p}}},
$$
where $\varepsilon = \varepsilon(n,p,\Lambda_{F}, 0, \alpha , \bar{R})$ is just the constant $\varepsilon$ in Theorem \ref{myers} with $\vartheta=0$.
Then, by (\ref{ric1}), when $\overline{\left\|{\rm{Ric}}_{\infty}^{\Lambda_{F}^{2}}\right\|}_{p, R_{0}, 0}\leq\varepsilon_{2}$, we have $\overline{\left\|{\rm{Ric}}_{\infty}^{\Lambda_{F}^{2}}\right\|}_{p, \bar{R}, 0}\leq \varepsilon$.
Consequently, it follows from Theorem \ref{myers} that
$$
{\rm{diam}}(M)\leq \frac{\Lambda_{F}^{-1}(\Lambda_{F}+1)\pi}{2}.
$$

In sum, for $0< R_{0} \leq \Lambda_{F}\pi$, let
$$
\varepsilon_{0}(n,p,\Lambda_{F},\alpha, R_{0}):=\min\{\varepsilon_{1}(n,p,R_{0}),\varepsilon_{2}(n,p,\Lambda_{F},\alpha)\}.
$$
Then, when $\overline{\left\|{\rm{Ric}}_{\infty}^{\Lambda_{F}^{2}}\right\|}_{p, R_{0}, 0}\leq \varepsilon_{0}$, we have (\ref{diaest-2}), and then, $M$ is compact.

$\mathbf{Step \ 2.}$ \   Now we will show that the universal covering space $\widetilde{M}$ of $M$ is compact.

We first claim that $(\widetilde{M},\widetilde{F})$ is forward complete. In fact, let $\tilde{\gamma}(t)$ be any geodesic of $(\widetilde{M},\widetilde{F})$ with $\tilde{\gamma}(0)=\tilde{x}_{0}\in\widetilde{M}$. Since $f$ is a local isometry, its projection $\gamma(t):=f\circ\tilde{\gamma}(t)$ is a geodesic of $(M,F)$ with $\gamma(0)=f(\tilde{x}_{0})=x_{0}\in M$. By the forward completeness of $(M,F)$, $\gamma(t)$ can be extended to a geodesic defined on $[0,\infty)$. The local isometry property of $f$ now implies that $\tilde{\gamma}(t)$ is also defined for all $t\in[0,\infty)$. Hence the universal cover $(\widetilde{M},\widetilde{F})$ is forward complete.

Next, we give the proof under two cases.

\par {\it Case 1}:  $R_{0}>\Lambda_{F}\pi$.  In this case, by Theorem \ref{myers}, there exists a nonnegative constant $\varepsilon = \varepsilon (n, p,\Lambda_{F}, 0, \alpha , R_{0})$, such that, when $\overline{\left\|{\rm{Ric}}_{\infty}^{\Lambda_{F}^{2}}\right\|}_{p, R_{0}, 0} \leq \varepsilon$,
we have ${\rm{diam}}(M)< R_{0}$. For our aim, we need to show that $\overline{\left\|\widetilde{{\rm{Ric}}}_{\infty}^{\Lambda_{F}^{2}}\right\|}_{p, R_{0}, 0}$ is controlled by $\overline{\left\|{\rm{Ric}}_{\infty}^{\Lambda_{F}^{2}}\right\|}_{p, R_{0}, 0}$.
In fact, for any point $\tilde{x}_{0}\in \widetilde{M}$, let $N$ be the minimal number of the fundamental domains $\rho(\Omega)$ covering $\widetilde{B}^{+}_{R_{0}}(\tilde{x}_{0})$, where $\widetilde{B}^{+}_{R_{0}}(\tilde{x}_{0})$ denotes the forward geodesic ball of radius $R_0$ with center at $\tilde{x}_{0}$ in $\widetilde{M}$. Then we have
\be
\widetilde{B}^{+}_{R_0}(\tilde{x}_0) \subset \bigcup^{N}_{i=1} \rho_{i}(\Omega) \subset \widetilde{B}_{2R_{0}}^{+}(\tilde{x}_0). \label{volB}
\ee
From (\ref{volB}), we can see that $\widetilde{m}(\widetilde{B}^{+}_{R_{0}}(\tilde{x}_{0}))\leq N\cdot m(M)\leq \widetilde{m}(\widetilde{B}^{+}_{2R_{0}}(\tilde{x}_{0}))$.
On the other hand,  by (\ref{doubvol}), we have
\be
\frac{\widetilde{m}\left(\widetilde{B}_{2R_0}^{+}(\tilde{x}_0)\right)}{\widetilde{m}\left(\widetilde{B}_{R_0}^{+}(\tilde{x}_0)\right)}\leq 2^{n+1}e^{\frac{2(K+ \vartheta_{0}^{2})}{3}R_{0}^2}. \label{volcom-3}
\ee
Then, from (\ref{volB}) and (\ref{volcom-3}),  we have the following
\beqn
\frac{1}{\widetilde{m}\left(\widetilde{B}_{R_0}^{+}(\tilde{x}_0)\right)} \int_{\widetilde{B}^{+}_{R_{0}}(\tilde{x}_{0})} \left(\widetilde{{\rm{Ric}}}_{\infty}^{\Lambda_{F}^{2}}\right)^{p} d \widetilde{m}
&\leq &
\frac{N}{\widetilde{m}\left(\widetilde{B}_{R_0}^{+}(\tilde{x}_0)\right)} \int_{M} \left({\rm{Ric}}_{\infty}^{\Lambda_{F}^{2}}\right)^{p} d m\\
&=& \frac{N}{\widetilde{m}\left(\widetilde{B}_{2R_0}^{+}(\tilde{x}_0)\right)} \cdot \frac{\widetilde{m}\left(\widetilde{B}_{2R_0}^{+}(\tilde{x}_0)\right)}{\widetilde{m}\left(\widetilde{B}_{R_0}^{+}(\tilde{x}_0)\right)}\int_{M} \left({\rm{Ric}}_{\infty}^{\Lambda_{F}^{2}}\right)^{p} d m\\
& \leq & \frac{2^{n+1} N e^{\frac{2(K+ \vartheta_{0}^{2})}{3}R_{0}^2}}{\widetilde{m}\left(\widetilde{B}_{2R_{0}}^{+}(\tilde{x}_0)\right)} \int_{M} \left({\rm{Ric}}_{\infty}^{\Lambda_{F}^{2}}\right)^p d m\\
& \leq & \frac{2^{n+1}e^{\frac{2(K+ \vartheta_{0}^{2})}{3}R_{0}^2}}{m(M)} \int_{M} \left({\rm{Ric}}_{\infty}^{\Lambda_{F}^{2}}\right)^p d m \\
& \leq & 2^{n+1}e^{\frac{2(K+ \vartheta_{0}^{2})}{3}R_{0}^2} \overline{\left\|{\rm{Ric}}^{\Lambda_{F}^{2}}_{\infty}\right\|}^{p}_{p, R_{0}, 0},
\eeqn
which implies that
\beq\label{epsilon2}
\overline{\left\|\widetilde{{\rm{Ric}}}^{\Lambda_{F}^{2}}_{\infty}\right\|}_{p, R_0, 0} \leq 2^{\frac{n+1}{p}} e^{\frac{2(K+ \vartheta_{0}^{2})}{3p}R_{0}^2} \overline{\left\|{\rm{Ric}}^{\Lambda_{F}^{2}}_{\infty}\right\|}_{p, R_0, 0}.
\eeq
Now, let $\bar{\varepsilon}_{0}:=\min \{\varepsilon, \ {2^{-\frac{n+1}{p}}}e^{-\frac{2(K+ \vartheta_{0}^{2})}{3p}R_{0}^2} {\varepsilon}\} =  {2^{-\frac{n+1}{p}}}e^{-\frac{2(K+ \vartheta_{0}^{2})}{3p}R_{0}^2} {\varepsilon}$, where $\varepsilon = \varepsilon(n,p,\Lambda_{F}, 0, \alpha , R_{0})$ is the constant $\varepsilon$ in Theorem \ref{myers} with $\vartheta=0$. It is not difficult to see that, when $\overline{\left\|{\rm{Ric}}^{\Lambda_{F}^{2}}_{\infty}\right\|}_{p, R_{0}, 0}\leq \bar{\varepsilon}_{0}$, we have $\overline{\left\|\widetilde{{\rm{Ric}}}^{\Lambda_{F}^{2}}_{\infty}\right\|}_{p, R_0, 0}\leq \varepsilon$. Hence compactness of $\widetilde{M}$ now follows from Theorem \ref{myers}.

\par {\it Case 2}:  $R_{0}\leq \Lambda_{F} \pi$.  In this case,  fix a constant $\bar{R}>\Lambda_{F}\pi$ and choose arbitrarily a point $\tilde{x}_{0}\in \widetilde{M}$ with $x_{0}=f(\tilde{x}_{0})\in M$. By Theorem \ref{myers}, there exists a nonnegative constant $\varepsilon = \varepsilon (n, p,\Lambda_{F}, 0, \alpha , \bar{R})$, such that, when $\overline{\left\|{\rm{Ric}}_{\infty}^{\Lambda_{F}^{2}}\right\|}_{p, \bar{R}, 0} \leq \varepsilon$,
we have ${\rm{diam}}(M)< \bar{R}$. Similar to the argument for (\ref{epsilon2}), we can obtain
\be
\overline{\left\|\widetilde{{\rm{Ric}}}^{\Lambda_{F}^{2}}_{\infty}\right\|}_{p, \bar{R}, 0} \leq 2^{\frac{n+1}{p}} e^{\frac{2(K+ \vartheta_{0}^{2})}{3p}\bar{R}^2} \overline{\left\|{\rm{Ric}}^{\Lambda_{F}^{2}}_{\infty}\right\|}_{p, \bar{R}, 0}.  \label{epsilon3}
\ee
On the other hand, following the argument for (\ref{est3}),  we have
\[
\overline{\left\|{\rm{Ric}}^{\Lambda_{F}^{2}}_{\infty}\right\|}^{p}_{p,\bar{R},0}(x_{0})\leq \sum_{i \in I} \frac{m\left(B_{R_{0}}^{+}(x_i)\right)} {m\Big(B_{\frac{R_{0}}{5\Lambda_{F}}}^{+}(x_i)\Big)} \cdot \frac{m\Big(B_{\frac{R_{0}}{5\Lambda_{F}}}^{+}(x_i)\Big)}{m\left(B_{\bar{R}}^{+}(x_0)\right)} \overline{\left\|{\rm{Ric}}^{\Lambda_{F}^{2}}_{\infty}\right\|}^{p}_{p,R_{0},0}(x_{i}).
\]
By Lemma \ref{chfen}, we obtain
\[
\frac{m\left(B_{R_{0}}^{+}(x_i)\right)}{m\Big(B_{\frac{R_{0}}{5\Lambda_{F}}}^{+}(x_i)\Big)} \leq (5\Lambda_{F})^{n+1}e^{\frac{K+\vartheta_{0}^{2}}{6}R_{0}^{2}}.
\]
Then we get
\[
\overline{\left\|{\rm{Ric}}_{\infty}^{\Lambda_{F}^{2}}\right\|}^p_{p, \bar{R}, 0} (x_{0})\leq (5 \Lambda_{F})^{n+1}  e^{\frac{K+\vartheta_{0}^{2}}{6}R_{0}^{2}} \ \overline{\left\|{\rm{Ric}}^{\Lambda_{F}^{2}}_{\infty}\right\|}_{p, R_{0}, 0}^{p},
\]
from which, we have
\be
\overline{\left\|{\rm{Ric}}_{\infty}^{\Lambda_{F}^{2}}\right\|}_{p, \bar{R}, 0} \leq (5 \Lambda_{F})^{\frac{n+1}{p}}  e^{\frac{K+\vartheta_{0}^{2}}{6p}R_{0}^{2}} \ \overline{\left\|{\rm{Ric}}^{\Lambda_{F}^{2}}_{\infty}\right\|}_{p, R_{0}, 0}. \label{ric2}
\ee
Now, from (\ref{epsilon3}) and (\ref{ric2}), we obtain the following
\be
\overline{\left\|\widetilde{{\rm{Ric}}}^{\Lambda_{F}^{2}}_{\infty}\right\|}_{p, \bar{R}, 0} \leq (10 \Lambda_{F})^{\frac{n+1}{p}} \ e^{\frac{2(K+ \vartheta_{0}^{2})}{3p}\bar{R}^2} \  e^{\frac{K+\vartheta_{0}^{2}}{6p}R_{0}^{2}} \ \overline{\left\|{\rm{Ric}}^{\Lambda_{F}^{2}}_{\infty}\right\|}_{p, R_{0}, 0}. \label{ric3}
\ee
In the  following, we take $\bar{R}=2\Lambda_{F}\pi$ and let
$$
\bar{\varepsilon}_{0}=\min \{\varepsilon , \, (10 \Lambda_{F})^{- \frac{n+1}{p}} \ e^{-\frac{8\Lambda_{F}^{2}\pi^{2}(K+ \vartheta_{0}^{2})}{3p}} \  e^{-\frac{K+\vartheta_{0}^{2}}{6p}R_{0}^{2}}\varepsilon\}= (10 \Lambda_{F})^{- \frac{n+1}{p}} \ e^{-\frac{8\Lambda_{F}^{2}\pi^{2}(K+ \vartheta_{0}^{2})}{3p}} \  e^{-\frac{K+\vartheta_{0}^{2}}{6p}R_{0}^{2}}\varepsilon .
$$
From (\ref{ric3}), it is obvious that, when $\overline{\left\|{\rm{Ric}}^{\Lambda_{F}^{2}}_{\infty}\right\|}_{p, R_{0}, 0}\leq \bar{\varepsilon}_{0}$, we have $\overline{\left\|\widetilde{{\rm{Ric}}}^{\Lambda_{F}^{2}}_{\infty}\right\|}_{p, \bar{R}, 0}\leq \varepsilon$. By Theorem \ref{myers} again,
we conclude that $\widetilde{M}$  is compact.

Finally, because of the compactness of $\widetilde{M}$, $\pi_{1}(M)$ is finite.
\end{proof}

\vskip 5mm

\noindent Xinyue Cheng \\
School of Mathematical Sciences \\
Chongqing Normal University \\
Chongqing, 401331, P.R. China\\
{\it E-mail address}: \ chengxy@cqnu.edu.cn

\vskip 3mm
\noindent Liulin Liu\\
School of Mathematical Sciences \\
Chongqing Normal University \\
Chongqing, 401331, P.R. China\\
{\it E-mail address}: \ Liuliulin99@126.com


\vskip.1in
\begin{thebibliography}{Ma}

\baselineskip 10pt
\bibitem{Aubry} Aubry, E.: Finiteness of $\pi_1$ and geometric inequalities in almost positive Ricci curvature. Ann. Sci. Ecole Norm. Sup., {\bf 40}, 675-695 (2007)

\bibitem{BaoChernShen} Bao, D., Chern, S.S., Shen, Z.: An Introduction to Riemann-Finsler Geometry. GTM 200, Springer, New York (2000)

\bibitem{ChF1} Cheng, X.,  Feng, Y.: Harnack inequality and the relevant theorems on Finsler metric measure manifolds. Results in Mathematics, {\bf 79},  166(2024)

\bibitem{ChF} Cheng, X., Feng, Y.: On Finsler metric measure manifolds with integral weighted Ricci curvature bounds. Sci. China Math., {\bf 69}, 461-486 (2026)

\bibitem{ChSh} Cheng, X., Shen, Z.: Some inequalities on Finsler manifolds with weighted Ricci curvature bounded below. Results Math., {\bf 77}, 70 (2022)

\bibitem{Chern} Chern, S.S.: Finsler geometry is just Riemannian geometry without the quadratic restriction. Notices Amer. Math. Soc., {\bf 43}(9), 959-963 (1996)

\bibitem{ChernShen}  Chern, S.S., Shen, Z.: Riemann-Finsler Geometry. Nankai Tracts in Mathematics, Vol. 6, Singapore: World Scientific, (2005)

\bibitem{Grov} Gromov M.:  Structures m\'{e}triques pour les vari\'{e}t\'{e}s riemanniennes. Paris: Cedic/Fernand Nathan (1981)

\bibitem{LiWuYu} Li, F., Wu, J., Zheng, Y.: Myers' type theorem for integral Bakry-\'{E}mery Ricci tensor bounds. Results Math., {\bf 76}, 32 (2021)

\bibitem{Myers} Myers, S.B.: Riemannian manifolds with positive mean curvature. Duke Math. J., {\bf 8}(2), 401-404 (1941)

\bibitem{Ohta0} Ohta, S.:  Finsler interpolation inequalities, Calc. Var. Partial Differential Equations, {\bf 36}(2009), 211-249.

\bibitem{Ohta1} Ohta, S.: Comparison Finsler Geometry. Springer Monographs in Mathematics, Springer, Cham (2021)

\bibitem{Ohta3} Ohta, S., Sturm, K.T.: Bochner-Weizenb\"{o}ck formula and Li-Yau estimates on Finsler manifolds. Adv. Math., {\bf 252}, 429-448 (2014)

\bibitem{PeterSprouse} Petersen, P., Sprouse, C.: Integral curvature bounds, distance estimates and applications. J. Differ. Geom., {\bf 50}(2), 269-298 (1998)

\bibitem{PeterWei}Petersen, P., Wei, G.: Relative volume comparison with integral curvature bounds. Geom. Funct. Anal., {\bf 7}(6), 1031-1045 (1997)

\bibitem{PeterW2} Petersen P, Wei G. Analysis and geometry on manifolds with integral Ricci curvature bounds II. Trans Amer Math Soc, 2000, 353(2): 457--478

\bibitem{Ra} Rademacher, H.B.: Nonreversible Finsler metrics of positive flag curvature, In: ``A Sampler of Riemann-Finsler Geometry", MSRI Publications, {\bf 50}, Cambridge: Cambridge University Press, (2004)

\bibitem{shen} Shen, Z.: Volume comparison and its applications in Riemann-Finsler geometry. Adv. Math., {\bf 128}(2), 306-328 (1997)

\bibitem{Shen1} Shen, Z.: Lectures on Finsler Geometry. Singapore: World Scientific, (2001)

\bibitem{Spro} Sprouse, C.: Integral curvature bounds and bounded diameter. Commnu. Anal. Geom., {\bf 8}(3), 531-543(2000)

\bibitem{WeiWylie} Wei, G., Wylie, W.: Comparison geometry for the Bakry-Emery Ricci tensor. J. Differ. Geom., {\bf 83}, 377-406 (2009)

\bibitem{WuXin} Wu, B.Y., Xin, Y.L.: Comparison theorems in Finsler geometry and their applications. Math. Ann., {\bf 337}(1), 177-196 (2007)

\bibitem{WuJY1} Wu, Jia-Y.: Comparison geometry for Integral Bakry-\'{E}mery Ricci tensor bounds. J. Geom. Anal., {\bf 29}, 828-867 (2019)

\bibitem{WuJY2} Wu, Jyh-Y.: Complete manifolds with a little negative curvature. Amer. J. Math., {\bf 113}(4), 567-572 (1991)

\bibitem{Zhao} Zhao, W.: Integral curvature bounds and diameter estimates on Finsler manifolds. Sci. China Math., {\bf 64}(3), 573-588 (2021)

\end{thebibliography}
\end{document}